\documentclass{article}
\usepackage{authblk}
\usepackage[margin=1in,footskip=0.33in]{geometry}
\usepackage{xr-hyper}
\usepackage{longtable,float,color,amsmath,amsxtra,amssymb,latexsym,amscd,amsthm,amsfonts,makecell,bbm}
\usepackage{hyperref}
\usepackage{multirow, booktabs, makecell}
\allowdisplaybreaks
\usepackage{footnote}
\usepackage{tabularx,array}
\usepackage{natbib}

\usepackage{graphicx, wrapfig, caption, adjustbox, fancyhdr}
\usepackage[dvipsnames]{xcolor}  
\usepackage{tikz}
\usetikzlibrary{positioning, shapes.multipart, shadows, calc}
\usepackage[all]{xy}
\usetikzlibrary{arrows.meta}
\usepackage[margin=1in]{geometry}

\newcolumntype{C}[1]{>{\centering\arraybackslash}m{#1}}

\newif\iffinalversion
\finalversiontrue 

\usepackage{changes, cancel, soul}
\usepackage{lineno}
\usepackage{xparse}

\iffinalversion
  \renewcommand{\added}[2][]{#2}   
  \renewcommand{\deleted}[2][]{}   
  \renewcommand{\replaced}[3][]{#2} 
  \renewcommand{\cancel}[1]{} 
  \NewDocumentEnvironment{change}{o}
    {\begingroup\IfValueTF{#1}{\color{#1}}{\color{black}}}
    {\endgroup}
\else
  \NewDocumentEnvironment{change}{o}
    {\begingroup\IfValueTF{#1}{\color{#1}}{\color{blue}}}
    {\endgroup}
\fi

\numberwithin{equation}{section}
\theoremstyle{plain}
\newtheorem{thm}{Theorem} 

\newtheorem{defn}[thm]{Definition} 
\newtheorem{rmk}[thm]{Remark}
\newtheorem{lmm}[thm]{Lemma}
\newtheorem{prop}[thm]{Proposition}
\newtheorem{cor}[thm]{Corollary}
\title{Understanding the invader-driven replicator dynamics}
\author[1]{Thi Minh Thao Le}
\author[2,1]{Marina Garcia-Romero}
\author[1]{João Duarte Alcântara Galvão}
\author[3]{Sten Madec}
\author[1] {Erida Gjini}
\affil[1]{Center for Computational and Stochastic Mathematics, Instituto Superior Técnico, Lisbon, Portugal}
\affil[2]{Universitat Politècnica de Catalunya, Barcelona, Spain}
\affil[3]{Institut Denis Poisson, University of Tours, Tours, France}
\date{}

\begin{document}

\maketitle
\begin{abstract}
In this paper, we study a special case of the invasion fitness matrix in a replicator equation: the invader-driven case. In this replicator, each species is defined by its unique active invasiveness potential (initial growth rate when rare), upon invading any other species, independently of the partner. We derive explicit expressions and theorems to fully characterize the steady-states of this system, including its unique interior coexistence regime, reached for positive species traits, or alternative boundary exclusion states, reached for negative species traits. We study the internal stability of coexistence steady-states, and the system’s stability to outsider invasion, relevant for system assembly. We provide detailed analytical results for critical diversity thresholds, and for the special case of random uniform species traits, we analytically compute the probability of stable $k-$ species coexistence in a random pool of size $N$, and show that the mean number of co-existing species can be approximated as $\mathbb{E}[n]\sim \sqrt{2N}$. We also derive explicit mathematical conditions for invader traits and invasion outcomes (augmentation, rejection, and replacement), dependent on the history of system assembly. Finally, by outlining links of this replicator case \added{with 3 systems: the old known competition model for self-reproducing macromolecules, }the corresponding (rank-1) Lotka-Volterra ecological systems, and certain epidemiological multi-strain SIS models with coinfection, we highlight the relevance of applying these mathematical principles to improve the theoretical and empirical understanding of multi-species coexistence \added{for ecology, epidemiology and biochemistry applications}. 
\end{abstract}
\section{Introduction}
The replicator dynamics has been extensively studied and used for its wide applicability in areas such as economics, ecology, epidemiology and microbiology \citep{Hofbauer_Sigmund_1998, cressman2014replicator, yoshino2008rank, chawanya2002large, gjini2020key, le2025inference,ferreira2025unpacking}. The replicator equation constitutes the foundation for modelling the interaction between species (also sometimes called `players' or ´strategies´ or `strains') for survival and dominance in a system, based on how fit they are relative to each other. In the particular instance of the replicator equation that we consider here, under classical frequency-dependence and linear fitness, the replicator equation is written in terms of an invasion fitness matrix (a special form of payoff), with invasion defined pairwise $\lambda_i ^j$, and the diagonal consisting of zeros \citep{madec2020predicting,le2023quasi}. It is described by the $N$-dimensional replicator system:
\begin{equation}\label{eq:replicator}
	\frac{dz_i}{dt }  = {\Theta z_i \cdot\bigg( \sum_{j\neq i} \lambda_i^j z_j -\mathop{\sum}_{1\leq k<j\leq N} (\lambda_j^k+\lambda_k^j) z_jz_k \bigg)},\quad i=1,\cdots,N
	\,,
\end{equation}
where $\lambda_i^j$'s denote pairwise invasion fitnesses between any two species \citep{geritz1998evolutionarily} and $\Theta$ gives the speed of the dynamics. In this equation the linear term inside the parenthesis $\sum_{j\neq i} \lambda_i^j z_j$ gives the mean invasion fitness of each species $i$ as a sum over all pairwise interactions, and the quadratic term $$
Q=\mathop{\sum}_{1\leq k<j\leq N} (\lambda_j^k+\lambda_k^j) z_j z_k
=
\sum_{i,j} \lambda_i^j z_iz_j=\sum_{i=1}^N z_i f_i(\mathbf{z})\equiv \bar{f}(\mathbf{z})
$$
describes the mean invasion fitness of the system, a quantity that depends on all species pairs, is dynamic in time, and acts equally on each species growth rate (see also \citep{gjini2023towards}).

In this paper, we focus on one special case of such a system, where each species always has the same invasion fitness when interacting with other system members, and so different species are defined by their unique invasiveness trait, independently of their partner, considered also in \citep{madec2020predicting, gjini2023towards}. What kind of multispecies coexistence results from such invasion fitness matrix? What are its defining principles and limitations? This is the question we address here. 

\added{
The mathematical formulation and dynamical behavior of replicator equations have been extensively studied, particularly in the context of pre‑biotic evolution and molecular self‑organization. Early work by \cite{epstein1979competitive} introduced a generalized class of replicator-type systems to model the competitive coexistence of self‑reproducing macromolecules. This line of research, related to the replicator equation studied here, was subsequently formalized and significantly extended by \cite{hofbauer1981competition,Hofbauer_Sigmund_1998}, who established the global dynamical properties of these systems, including the existence of Lyapunov functions, the monotonicity of key quantities, and the resulting global stability of equilibria under broad structural conditions.}

While many ecological system models are based on the Lotka-Volterra formalism, to study ecosystem stability, complexity, composition, and invasion, the related but simpler replicator equation framework has been less readily used. Here, motivated by our previous work \citep{madec2020predicting,gjini2020key,gjini2023towards,le2023quasi}, and the relative elegance of this formalism, we approach the same questions of ecosystem composition, stability and invasibility through the lens of the replicator equation, and for this particular replicator equation case. Community ecology theory explains biological invasions through modern niche concepts, particularly the notion of niche opportunity, which encompasses resource availability, enemy pressure, and environmental variability \citep{shea2002community}. Research efforts are long ongoing to provide a predictive framework for understanding invasion dynamics and the role of disturbances, interactions, or system diversity \citep{hui2019invade,seebens2025biological,kurkjian2021impact,gjini2023towards}. In this work, by the very nature of the replicator equation, written in terms of invasion fitnesses, the invasion resistance of the community, invader traits and performance can be directly inter-related, although we show this relation is rather intricate mathematically.

Thus, our study of this special case of the replicator equation has implications beyond mathematical results, for understanding ecosystem composition, stability, assembly and resistance to invasion. There is currently much attention on biological invasions as a leading cause of biodiversity loss, with significant ecological, societal, and economic consequences \citep{seebens2025biological}. Preventing the introduction and spread of alien species is the most effective means of mitigating these impacts. Here, we contribute to this challenge by elucidating the mathematical basis of successful or unsuccessful invasions in this type of replicator system, which can be key for designing interventions in natural systems. Our work provides clear mathematical criteria for different invasion outcomes, and for quantifying niche opportunities and system assembly dynamics, beyond pairwise mechanisms \citep{levine2017beyond}. 

The outline of the paper is as follows. In Section 2, we present detailed results on the steady-states of such a replicator system, focusing further on the only regime leading to coexistence, i.e. under positive invasion fitness traits. In Section 3, we address the number of coexisting species. We provide an expression for the probability of $k$-species coexistence under the special case of random uniform $\lambda_i$. For this trait distribution, we also show that the mean number of coexisting species can be approximated as $\sqrt{2N}$ where $N$ is the pool size. In Section 4, we study random system assembly and provide detailed conditions for new invasion outcomes, with a long history of ecological interest \citep{case1990invasion}, and prove special regimes of gradual niche saturation. In Section 5, we link the invader-driven replicator equation to real biological systems, providing an explicit bridge with Lotka-Volterra systems \citep{Hofbauer_Sigmund_1998}, \added{competitive coexistence of self-replicating macromolecules \citep{epstein1979competitive,hofbauer1981competition},} and coinfection SIS models with many strains \citep{madec2020predicting,le2023quasi}, including data from \textit{Streptococcus pneumoniae} serotype colonization \citep{le2025inference}. We believe we provide a full journey into the fascinating world of this replicator equation special case, that can be useful across many fields, from evolutionary game theory to biology and community ecology.

\section{Invader-driven fitness matrix and $n$-species steady state}

Let us consider the case when species are defined by their `invasiveness' property, i.e. their propensity for growth in a given resident system, is completely independent of the encountered resident. This means $\lambda_i^j=\lambda_i, \forall j\neq i$ and we say that the matrix $\Lambda=(\lambda_i^j)_{i,j}$ is invader-driven.
\\
We denote the simplex
$$\Delta=\{\mathbf{z}\in[0,1]^N,\; \sum_{i=1}^N z_i=1\}.$$
\added{Without loss of generality, we can set $\Theta = 1$ up to a time-scaling $t \mapsto \Theta t$.}
When $\Lambda$ matrix is invader-driven, the replicator equation on $\Delta$ reads:
\begin{equation} \label{eq:repli_invader}
        \dfrac{dz_i}{d\tau} = z_i\left(\lambda_i\left(1 - z_i\right) -Q\right), \qquad 1\leq i\leq N.
\end{equation}
with
$$
Q = \sum_{i<j}^N\left(\lambda_i^j + \lambda_j^i\right)z_jz_i=\sum_{i\neq j} \lambda_{i}^j z_iz_j = \sum_{i=1}^N \lambda_i\left(1- z_i\right)z_i.
$$
Solving the replicator equation  \ref{eq:repli_invader} for obtaining the equilibrium $\mathbf{z^*}=(z_1^*,z_2^*,\cdots,z_N^*)\in\Delta$, we obtain for all $i$:
        \begin{equation}
        \label{invader-driven}
        z_i^*=0 \qquad\text{ or }\quad
        z_i^*=1-\frac{Q^*}{\lambda_i}
        \end{equation}
where
$$
Q^*=\sum_{i=1}^N \lambda_i\left(1- z_i^*\right)z_i^*.
$$
Hence, if species $i$ is present at equilibrium, its steady state frequency satisfies  $0\leq 1-z^*_i=  \frac{Q^*}{\lambda_i}$ and this implies that all the $\lambda_i$ corresponding to the coexisting species must have the same sign as $Q^*$. Assume that there is $k$-coexisting species,summing up all the frequencies yields to
\begin{equation}
        Q^* = \dfrac{k-1}{\sum_{i=1}^k \frac{1}{\lambda_i}},
        \label{eq:Qstar}
\end{equation} with $k$ is the number of persisting species at steady state.
\replaced{Before presenting the main results, we provide an illustration of this type of replicator dynamics in Figure \ref{fig:dynamics}. }{ We have the following proposition.} 

\begin{figure}[htb!]
    \centering
    \includegraphics[width=1\linewidth]{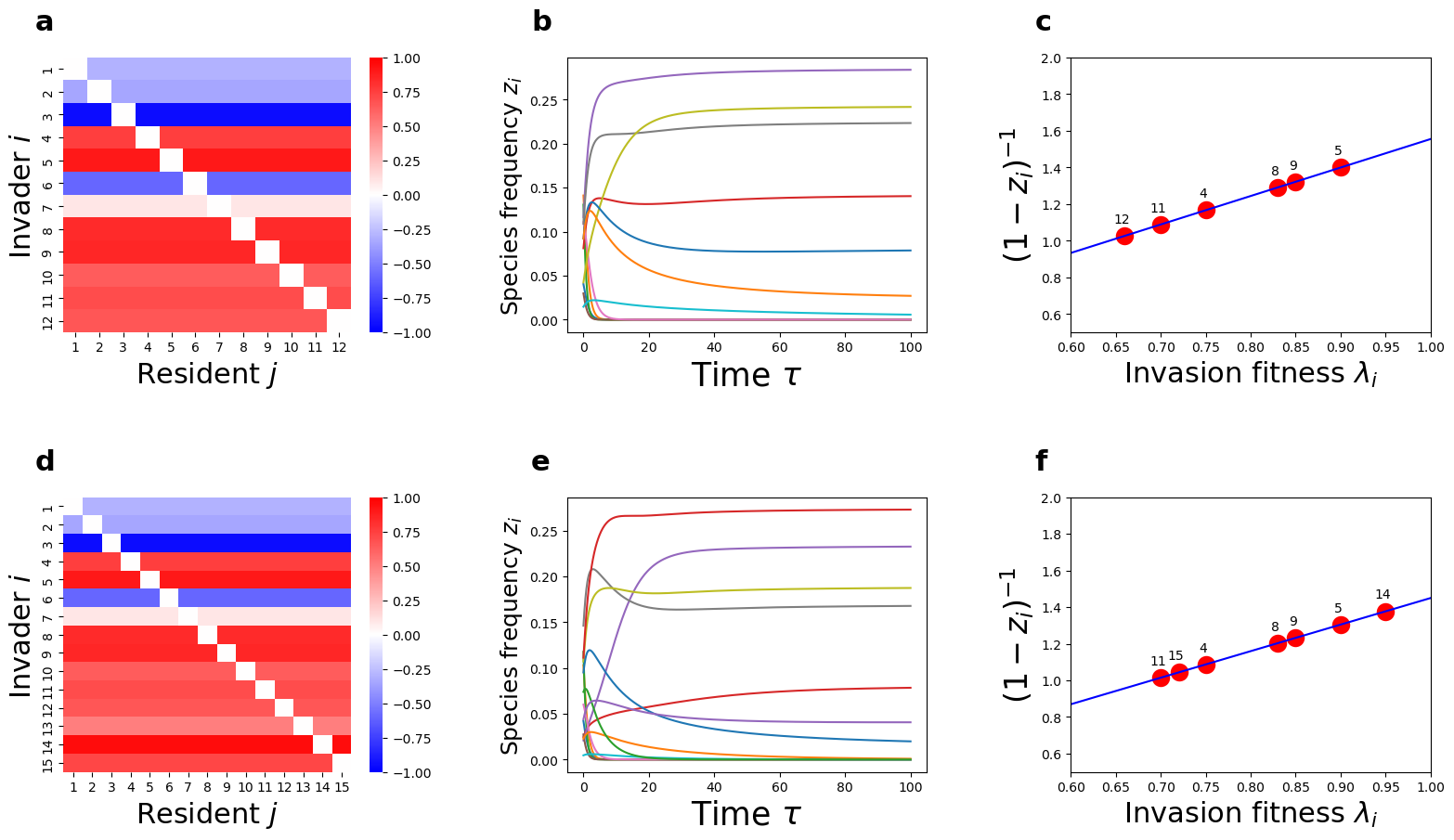}
    \caption{\textbf{Illustration of two cases of invader-driven replicator dynamics.} \added{On each row, we present a different system, simulating an invader-driven dynamics of 12 (a-c) and 15 (d-f) species, where system 2 (bottom row) is an augmented version of system 1 (top row), with the same 12 species as system 1, plus 3 additional species. Panels \textbf{(a)} and \textbf{(d)} Invasion fitness matrix $\Lambda = (\lambda_{i}^j)$, with $\lambda_{i}^j = \lambda_i\sim \mathcal{U}\left[-1,1\right]$ for $j\neq i$ and $\lambda^i_i = 0$ for all $i$. 
Panels \textbf{(b)} and \textbf{(e)} Replicator dynamics from an equal initial condition for all species, showing convergence to a coexistence equilibrium with a subset of \replaced{ species }{ speciess } persisting. 
Panels \textbf{(c)} and \textbf{(f)} Linearity between $\lambda_i$ and $\left(1 - z_i^*\right)^{-1}$ for all persistent species $i$. The surviving species are marked by red circles labeled via their index id. Note that the invasion fitness matrix in the second row extends the one in the first row; that is, the fitness values for species 1 through 12 are identical in both systems. Yet we see the set of coexisting species is different.}}
    \label{fig:dynamics}
\end{figure}

\subsection{Positive fitness wins}

We have the following proposition.

\begin{prop}\label{prop:positive}
Given a replicator system with invader-driven matrix $\Lambda$, denote the subset of species with positive invasion fitness by
$S^+ = \left\{ i \in \{1,\dots,N\} : \lambda_i > 0 \right\}$,
and the subset of species with non-positive invasion fitness by
$S^- = \left\{ i \in \{1,\dots,N\} : \lambda_i \leq 0 \right\}$.
Assume additionally that $S^+ \neq \emptyset$ and
$\sum_{i \in S^+} z_i(0) > 0$.
Then the set of all surviving species at an endemic steady state is a subset of $S^+$:
$$
        \forall i \in S^-:\quad \lim_{t\to+\infty} z_i(t) = 0.
$$
\end{prop}

\begin{proof}
Let $\mathbf{z}(t) \in \Delta$ be a solution of \eqref{eq:repli_invader}.Consider the function
$V(\mathbf{z}) = \sum_{i \in S^-} z_i$.        Clearly $V(\mathbf{z}(t)) \in [0,1]$ for all $t$. Note that
\begin{equation}\label{V=0or1}
V = 0 \;\Longleftrightarrow\; z_i = 0 \ \forall i \in S^-,
\quad\text{and}\quad
V = 1 \;\Longleftrightarrow\; z_i = 0 \ \forall i \in S^+.
\end{equation}
Therefore, it suffices to show that $V(\mathbf{z}(t)) \to 0$ as $t \to +\infty$. Using the identity
$$
Q = \sum_{i \in S^+} \lambda_i z_i(1 - z_i)
+ \sum_{i \in S^-} \lambda_i z_i(1 - z_i),
$$
we obtain
\begin{equation*}
\begin{aligned}
\frac{d}{dt} V(\mathbf{z}(t))
&
\,= \,
\sum_{i \in S^-} \frac{d}{dt} z_i
\,= \, 
\sum_{i \in S^-} z_i\bigl(\lambda_i (1 - z_i) - Q\bigr) 
\,=\, 
\sum_{i \in S^-} \lambda_i z_i (1 - z_i) - Q V(\mathbf{z}) \\
&
\,= \,
\biggl(\sum_{i \in S^-} \lambda_i z_i (1 - z_i)\biggr)\bigl(1 - V(\mathbf{z})\bigr)
- \biggl(\sum_{i \in S^+} \lambda_i z_i (1 - z_i)\biggr) V(\mathbf{z}) 
\qquad\le 0
\end{aligned}
\end{equation*}
Both term in this equation is non-positive so, by \eqref{V=0or1}, we have $\frac{d}{dt} V(\mathbf{z}(t)) = 0$ if and only if $V = 0$ or $V = 1$.

By LaSalle’s invariance principle applied to the Lyapunov function $V(\mathbf{z})$, the $\omega$-limit set of $\mathbf{z}(t)$ is contained in the largest invariant subset of $\{ \mathbf{z} : \dot V(\mathbf{z}) = 0 \}$, which consists of states with $V = 0$ or $V = 1$. Hence either $V(\mathbf{z}(t)) \to 0$ or $V(\mathbf{z}(t)) \to 1$.

Since $t \mapsto V(\mathbf{z}(t))$ is non-increasing and $V(\mathbf{z}(0)) = 1 - \sum_{i \in S^+} z_i(0) < 1$
by assumption, we have $V(\mathbf{z}(t)) < 1$ for all $t \ge 0$. Therefore the limit $V(\mathbf{z}(t))$ as $t \to +\infty$ cannot be $1$, and we conclude that
     \[
        V(\mathbf{z}(t)) \to 0 \quad \text{as } t \to +\infty.
        \]
This implies $z_i(t) \to 0$ for all $i \in S^-$, which completes the proof.
\end{proof}

\begin{rmk}
\deleted{According to the proof of Proposition \ref{prop:positive}, if $S^+$ is not empty, the invader-driven dynamics converges to an equilibrium.}
\end{rmk}
Considering that when $\lambda_i$ include both negative and positive entries, the selection proceed to keep only the positive-fitness species, we  from now on focus on the case, when the invasion traits are all positive, or all negative. The $\lambda_i-$ positive scenario is the only scenario that can yield coexistence of multiple species, that is why we present all the analysis of this case in the main text. The $\lambda_i$-negative case, which leads always to competitive exclusion, is detailed with formal proofs in the supplements (See Section \ref{supp:negative} and Figure \ref{fig:negative}).

\subsection{Identities of coexisting species and global convergence to equilibrium} \label{sec:main-result-dynamics}
\begin{rmk}
By the Proposition \ref{prop:positive}, from now on, if there is \replaced{no }{not} further assumption, without loss of generality, \replaced{we }{to } consider the coexistence in the case of invader-driven matrix with at least one positive fitness. It suffices to consider the case when all $\lambda_{i} >0$ for all $i$ and we can rearrange the name/order of \replaced{ species }{ speciess } such that $$\lambda_1 \geq \lambda_2 \geq \dots \geq\lambda_N.$$ 
\end{rmk}

\deleted{We set $E_k { = \left\{1,2,\dots,k\right\} }\subset \left\{1,2,\dots,N\right\}$, $k \leq N$ be the set of \replaced{ species }{ speciess } coexistence in the long run. According to the steady state frequency expression, it is trivial to see that, \added{at an equilibrium}, the hierarchy of relative abundances follows strictly the hierarchy of species invasion fitnesses:}
\begin{change}
Next, we prove global attraction of the unique equilibrium.
\begin{thm} \label{thm:main-converge}
The invader-driven dynamics converges to the stable equilibrium $z^*$ for all \added{strictly positive} initial states. Furthermore, $z^* = \left(z^*_1,z_2^*,\dots, z^*_k,0,\dots,0\right)$ satisfies $\lambda_{k+1} < Q^*_k < \lambda_k$ and $z^*_i = 1 -\frac{Q_k^*}{\lambda_i}$, for all $1 \leq i \leq k$, where $Q^*_k = \frac{k-1}{\sum_{i=1}^k \frac{1}{\lambda_i}}$.
\end{thm}

\begin{proof}

We define the potential
\begin{equation*}
V\left(z\right) := \sum_{i=1}^{N} \left(\lambda_iz_i - \frac{\lambda_i}{2}z_i^2\right)
\quad
\Longrightarrow \,f_i\left(z\right):=\frac{\partial V}{\partial z_i} = \lambda_i\left(1-z_i\right),\,\forall 1\leq i\leq N\,.
\end{equation*}
According to \cite[Theorem 19.5.1]{Hofbauer_Sigmund_1998}, \eqref{eq:repli_invader} are Shahshahani gradients on the simplex $\Delta = \{z\in \mathbb{R}^N:\, z_1 + \dots + z_N = 1\}$ with the potential $V$. Then, by \cite[Exercise 19.5.3]{Hofbauer_Sigmund_1998},
the potential $V$ is strictly concave on the simplex $\Delta$
and hence,
\eqref{eq:repli_invader} converges to the unique, globally attracting interior rest point $z^*$ (i.e. $z^*$ does not lie on the boundary of $\Delta$).

Now, it suffices to determine $z^*$ which maximizes $V$ on $\sum_{i=1}^N z_i = 1, \ z_i \ge 0$ for all $1\leq i\leq N$. For this purpose, we consider the Lagrangian
$$\mathcal{L}\left(z,\nu,\mu\right) = \sum_{i=1}^{N} \left(\lambda_iz_i - \frac{\lambda_i}{2}z_i^2\right)
-
\nu\left(\sum_{j=1}^{N}z_i - 1\right)
+ \sum_{i=1}^N \mu_i z_i
\,.$$
The KKT conditions give us
$$
\lambda_i(1 - z_i) = \nu - \mu_i, \qquad \mu_i \geq 0, \qquad \mu_i\, z_i = 0,\quad \forall 1\leq i\leq N.
$$
If $\nu < \lambda_i$, $z_i - \frac{\mu_i}{\lambda_i} = 1 - \frac{\nu}{\lambda_i} \geq 0$.
Thus, if $z_i = 0$, $-\frac{\mu_i}{\lambda_i} > 0$ i.e. $\mu_i < 0$, which is absurd. This leads to $z_i > 0$, implying $\mu_i = 0$ then $z_i=1 -\frac{\nu}{\lambda_i}$.\\
If $\nu \geq \lambda_i$, $z_i - \frac{\mu_i}{\lambda_i} = 1 - \frac{\nu}{\lambda_i} \leq 0$, yielding to $z_i \leq \frac{\mu_i}{\lambda_i}$. Since $\lambda_i, z_i, \mu_i$ are non-negative and $\mu_iz_i =0$, then $z_i = 0$.
\\
Hence, for all $1\leq i\leq N$, we have that
$$
z_i = \max\left\{ 1 - \frac{\nu}{\lambda_i}, 0 \right\}, \quad \sum_{i=1}^N z_i = 1.
$$
This implies $\nu \geq 0$ since $0 \leq z_i \leq 1$ for all $1\leq i \leq N$. Trivially, $\nu \neq 0$, because, if $\nu = 0$ then $z_i = 1$ for all $1\leq i \leq N$, which is absurd, due to $N\geq 2$.

With the order $\lambda_1 \ge \cdots \ge \lambda_N$, according to the uniqueness of interior rest point $z^*$, the unique $\nu$ satisfying the constraint lies strictly between $\lambda_{k+1}$ and $\lambda_k$,
and equals $\frac{k-1}{\sum_{i=1}^k \frac{1}{\lambda_i}}$. This yields exactly $z^*= \left(z^*_1,z_2^*,\dots, z^*_k,0,\dots,0\right)$ with $z^*_i = 1 -\frac{Q_k^*}{\lambda_i}$, for all $1 \leq i \leq k$ and $Q^*_k = \frac{k-1}{\sum_{i=1}^k \frac{1}{\lambda_i}}$. Strict concavity leads to the uniqueness, i.e.
$V(z) \le V(z^*)$ with equality if and only if $z = z^*$.
\\
Secondly, differentiate $V$ along trajectories (recalling that $\sum_{i=1}^N z_i = 1$ and $Q(z) = \sum_{i=1}^N z_i f_i(z)$), we have that
$$
\begin{aligned}
\frac{d}{dt}{V}(z) 
= 
\sum_{i=1}^N \frac{\partial V}{\partial z_i} \frac{d}{dt}{z}_i
= \sum_{i=1}^N f_i(z) z_i \left(f_i(z) - Q(z)\right)
&
= 
\sum_{i=1}^N z_i \left(f_i^2(z) - 2f_i(z)Q(z)   \right)
+ Q(z)\sum_{i=1}^N z_i f_i(z)
\\
&
=
\sum_{i=1}^N  z_i \left(f_i(z) - Q(z)\right)^2
\geq 0.
\end{aligned}
$$
We have $\frac{d}{dt}{V}(z) = 0$ if and only if $f_i(z) = Q(z)$ for all indices $i$ satisfying $z_i > 0$.
Since $f_i(z) = \lambda_i (1 - z_i)$, this gives
$$
z_i = 1 - \frac{Q(z)}{\lambda_i} \quad \text{whenever} \quad z_i > 0,
$$
and $z_i = 0$ otherwise. This is the same threshold form as before with
$\nu = Q(z)$. 

Therefore, \added{$V(z) \leq V(z^*)$ for all $z \neq z^*$ and $\frac{d}{dt}{V} \geq 0$}, $
\frac{d}{dt}{V} = 0$ if and only if $z = z^*$, then $z^*$ is the globally attracting rest point of the invader-driven dynamics with strictly positive fitnesses.

\end{proof}

Prior to the next result, we recall that a monomorphic state of a dynamical system is a fixed point located at a vertex of the standard simplex, defined by the state vector $x = e_k$, where the $k$-th component equals one and all other components equal zero.
\begin{prop} \label{prop:1unstable}
Every monomorphic (competitive-exclusion) equilibrium is unstable.
\end{prop}
This remark is a trivial consequence of Theorem \ref{thm:main-converge}, since $Q^*_1(z) = 0$ and $\lambda_i > 0 $ for all $1\leq i\leq N$. It can be also deduced by computing directly the Jacobian at $e_1 = (1,0,\dots,0)$.
\\ \\
The following Figure \ref{fig:pattern_lambdas_schema} illustrates a simple algorithm to determine the surviving species, by computing $Q^*_i$, starting $i=1,2,\dots$ and compare if $Q^*_i \in \left(\lambda_{i+1}, \lambda_i\right)$.
\begin{figure}[htb!]
    \centering
\begin{tikzpicture}
  \draw[very thick] (0,0) -- (12,0); 
  \foreach \x/\label in {0.0/\textbf{0}, 1/\textbf{1}} {
    \pgfmathsetmacro{\xpos}{12*\x}
    \draw[very thick] (\xpos,0.15) -- (\xpos,-0.15);
    \node[below] at (\xpos,-0.15) {\label};}

  \foreach \x/\label in {0.07/$\boldsymbol{Q_2^*}$, 0.22/$\boldsymbol{Q_3^*}$, 0.40/$\boldsymbol{Q_{k-1}^*}$,  0.52/$\boldsymbol{Q_k^*}$} 
  {
    \pgfmathsetmacro{\xpos}{10*\x}
    \draw[very thick, red] (\xpos,0.15) -- (\xpos,-0.15);
    \node[below, red] at (\xpos,-0.15) {\label};
  }

  \foreach \x/\label in {0.63/$\boldsymbol{\lambda_k}$, 0.72/$\boldsymbol{\lambda_{k-1}}$, 0.93/$\boldsymbol{\lambda_3}$, 0.99/$\boldsymbol{\lambda_2}$, 1.13/$\boldsymbol{\lambda_1}$} {
    \pgfmathsetmacro{\xpos}{10*\x}
    \draw[very thick, OliveGreen] (\xpos,0.15) -- (\xpos,-0.15);
    \node[below, OliveGreen] at (\xpos,-0.15) {\label};
  }
  \node at (3.05,-0.5) {\textcolor{red}{\textbf{\dots}}}; 
   \node at (8.3,-0.5) {\textcolor{OliveGreen}{\textbf{\dots}}};
  \draw[->, very thick, color=red] (1.7,-1.2) -- (4.2,-1.2);
  \draw[->, very thick, color=OliveGreen] (9.8,-1.2) -- (7.3,-1.2);  
\end{tikzpicture}
\caption{\textbf{Invader-driven system assembly.} Diagram of the ecological assembly mechanism of $k$ species in an invader-driven system, with $0<\lambda_i<\lambda_{i-1}<\dots<\lambda_1$ and $i=1,\cdots,N$. As more species are added to the system, the mean invasion resistance $Q^*$ increases and a natural threshold emerges below which less fit species cannot coexist anymore in the invader-driven replicator.}
\label{fig:pattern_lambdas_schema}
\end{figure}
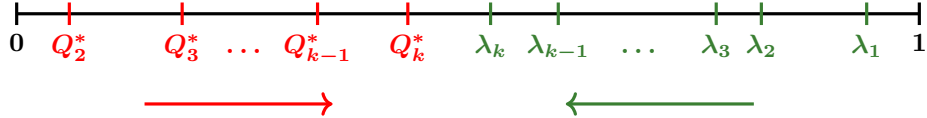

This algorithm is based on the following result.
\begin{prop} \label{rmk:prop-k}
Assume \eqref{eq:repli_invader} has $k$ species surviving in the long run and denote $Q^*_i = \frac{i-1}{\sum_{j=1}^i\frac{1}{\lambda_j}}$ for all $1\leq i\leq N$. Then $k$ is the first index satisfying the inequality $\lambda_{i+1} < Q_i^*$ and also is the unique index satisfying $\lambda_{i+1} < Q^*_i < \lambda_i$.
\end{prop}
\begin{proof}
According to Theorem \ref{thm:main-converge}, it suffices to prove that there exists unique $k$ such that $\lambda_{k+1} < Q^*_k < \lambda_k$. \\ \\
$\bullet$ \textbf{Existence.}
By the decreasing arrangement of $\lambda_{i}$, we can set a pseudo-$\lambda_{N+1}$ which equals $0$ and define the sequence $D_i:= Q_i^* - \lambda_{i+1}$ for all $1\leq i\leq N$. Then $D_1 = Q^*_1 - \lambda_2 = -\lambda_2 < 0$ and $D_N = Q^*_N - \lambda_{N+1} = Q^*_N > 0$. Thus, $\left\{D_i\right\}_{i\geq 1}$ starts \deleted{with} negative and ends \deleted{with} positive. Hence, we can define $D_{k_0}$ be the first non-negative term in the sequence, which \deleted{implies } \added{means that $D_{k_0} > 0$ and $D_{k_0-1} < 0$. By the definition of $D_i$, this leads to $Q^*_{k_0} > \lambda_{k_0 + 1}$ and $Q^*_{k_0-1} < \lambda_{k_0}$.}
\begin{equation} \label{eq:existence-system}
\cancel{\left\{
\begin{aligned}
& D_{k_0} > 0 ,
&&\Longrightarrow
Q^*_{k_0} > \lambda_{k_0 + 1}\\
& D_{k_0-1} < 0 ,
&&\Longrightarrow
Q^*_{k_0-1} < \lambda_{k_0}.
\end{aligned}
\right.}
\end{equation}
It suffices to prove $Q^*_{k_0} < \lambda_{k_0}$. Indeed, setting
\begin{equation}
S_{k} = \sum_{j=1}^{k}\frac{1}{\lambda_j}
\Longrightarrow
Q^*_{k_0 - 1} = \frac{k_0 - 2}{S_{k_0-1}}
\quad
\text{and}\;\;
Q^*_{k_0} = \frac{k_0 - 1}{S_{k_0-1} + \frac{1}{\lambda_{k_0}}}\,.
\end{equation}
This leads to
\begin{equation*}
Q^*_{k_0-1} < \lambda_{k_0}
\Longleftrightarrow
\frac{k_0 - 2}{S_{k_0-1}} < \lambda_{k_0}
\Longleftrightarrow
S_{k_0-1} > \frac{k_0 -2}{\lambda_{k_0}}\,.
\end{equation*}
Thus, we obtain that
\begin{equation*}
Q^*_{k_0} = \frac{k_0 - 1}{S_{k_0-1} + \frac{1}{\lambda_{k_0}}}
<
\frac{k_0-1}{\frac{k_0-2}{\lambda_{k_0}} + \frac{1}{\lambda_{k_0}}} = \lambda_{k_0}.
\end{equation*}
Hence, we have that $\lambda_{k_0+1} < Q^*_{k_0} < \lambda_{k_0}$. 
\\ \\
$\bullet$ \textbf{Uniqueness:} It suffices to prove that once $D_k$ turns positive, it stays positive. Indeed, assume $D_k > 0$ for some $1\leq k\leq N$, which means $Q^*_k > \lambda_{k+1}$. Since $Q^*_k = \frac{k-1}{S_k}$, we deduce $S_k < \frac{k-1}{\lambda_{k+1}}$. We consider
\begin{equation*}
Q^*_{k+1} = \frac{k}{S_k + \frac{1}{\lambda_{k+1}}}
>
\frac{k}{\frac{k-1}{\lambda_{k+1}} + \frac{1}{\lambda_{k+1}}}
=\lambda_{k+1}\,,
\end{equation*}
yielding to $Q^*_{k+1} > \lambda_{k+1} \geq \lambda_{k+2}$. Thus $D_{k+1} = Q^*_{k+1} - \lambda_{k+2} > 0$. By induction arguments, once the sequence becomes positive it never returns to non-positive values, and moreover $Q^*_{i} > \lambda_i$, for all $i > k$.
\end{proof}
\end{change}
In Figure \ref{fig:coexSS}
we show an illustration of this uniqueness property of the stable steady state, from different random simulations of invader-driven replicator dynamics. 

\begin{figure}[t!]
    \centering
    \includegraphics[width=0.95\linewidth]{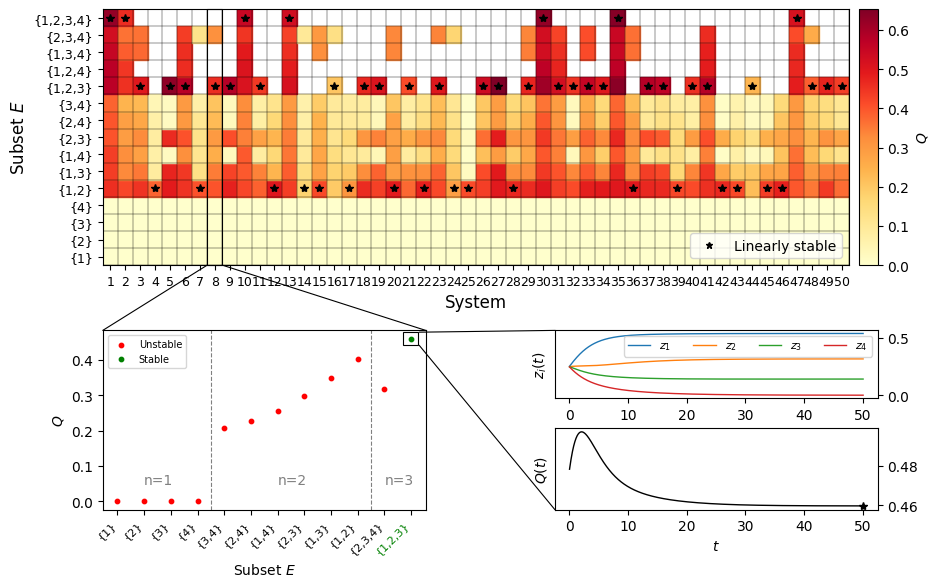}
    \caption{\textbf{Coexistence subsets $E$ and global mean fitness $Q$ at invader-driven replicator equilibria.} We examine all the feasible steady-states of $50$ different invader-driven systems with $\lambda_i\sim \mathcal{U}\,[0,1]$ $\forall\,i\in S$. The top heatmap shows $Q$ for all the equilibria subsets $E$ (rows) in each system (columns), colored by magnitude and marked with a black star if the equilibrium is linearly stable (computed numerically). Subsets are ordered from bottom to top by increasing $k$. The bottom panels show, for one selected system (highlighted in the grid) with $\lambda_1 = 1 \ ,\ \lambda_2 = 0.674  \ ,\ \lambda_3 = 0.536\ ,\ \lambda_4= 0.342$, on the left a scatter plot of $Q$ for the different equilibria of the system, and on the right the time evolution of the frequencies $z_i(t)$, $i\in S$ and the mean fitness $Q(t)$. In the scatter plot, the unstable equilibria are red dots, and the stable one, which corresponds with the maximum value of $Q$ within the system, is green. Moreover, the green label $\{1,\,2,\,3\}$ indicates the result from the particular dynamics simulation.}
    \label{fig:coexSS}
\end{figure}

\section{Random uniform $\lambda_i$ and probability of $k$ species coexistence}

Using system simulations for different pool size $N$ and random realizations of the $\lambda$ vector, we can see that stable coexistence of a large number of species is not easy in this replicator dynamics. Numerically we can obtain the probability distribution of the number of coexisting species (Figure \ref{fig:ndistpanel}), but the challenge is: can we obtain it formally?

\begin{figure}[hbt!]
    \centering
    \includegraphics[width=0.8\linewidth]{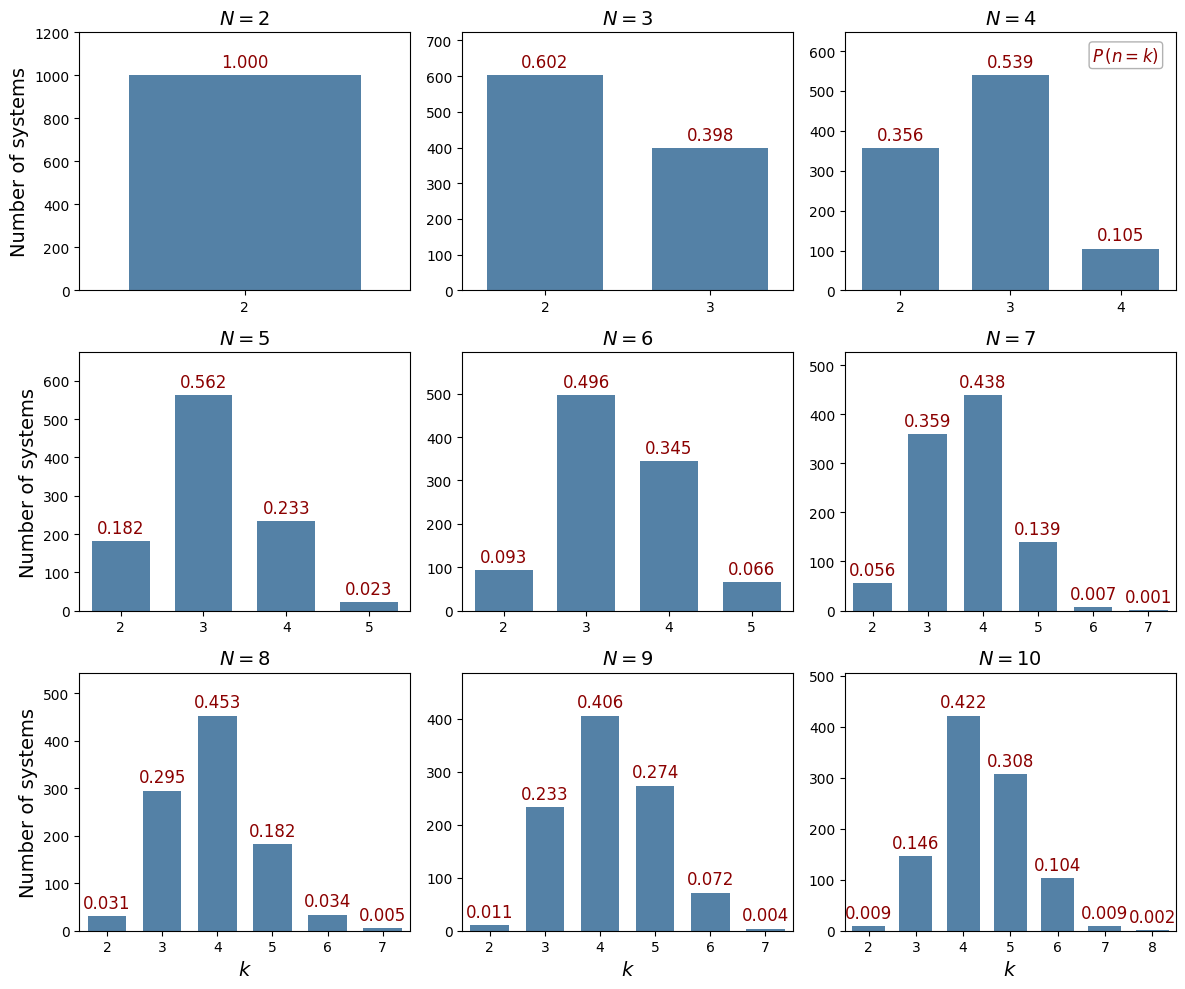}
    \caption{\textbf{Distribution of the number of coexisting species in the invader-driven replicator, for different species pool size $N$.} Summary of 1000 simulations for each $N$ between 2 and 10.}
    \label{fig:ndistpanel}
\end{figure}
\added{Initially, we recall that, order statistics are the sample values of a random variable, arranged in ascending order from smallest to largest.}
\\
Now, let $\lambda_{1} \ge \cdots \ge \lambda_{N}$ be the order statistics of i.i.d.\ $\lambda_i \sim \mathcal{U}[0,1]$.
Recalling to Proposition \ref{prop:1unstable}, for each $k \ge 2$, we recall the notation $Q^*_k = \frac{k-1}{\sum_{j=1}^k \frac{1}{\lambda_{j}}}$.
By Theorem \ref{thm:main-converge} , the number of surviving species, denoted by $n$ is \added{characterized by}
$$
m = k \quad \Longleftrightarrow \quad \lambda_{k+1} < Q^*_k < \lambda_{k}.
$$
We consider that
\begin{equation} \label{eq:prob-cond1}
Q_k^* < \lambda_k \quad \Longleftrightarrow
\quad
\sum_{j=1}^{k}\frac{1}{\lambda_j} > \frac{k-1}{\lambda_k}
\quad
\Longleftrightarrow
\quad
\lambda_k \sum_{j=1}^{k-1}\frac{1}{\lambda_j} > k-2\,,
\end{equation}
and
\begin{equation}\label{eq:prob-cond2}
Q_k^* > \lambda_{k+1} \quad \Longleftrightarrow
\quad
\sum_{j=1}^{k}\frac{1}{\lambda_j} < \frac{k-1}{\lambda_{k+1}}
\quad
\Longleftrightarrow
\quad
\lambda_{k+1} \sum_{j=1}^{k}\frac{1}{\lambda_j} < k-1\,.
\end{equation}
We now compute the probability of the event ``Exact $k$ species survive in the long run, i.e. \replaced{by Theorem \ref{thm:main-converge}, contribute to the stable equilibrium }{ contribute} in the stable equilibrium." base on the condition \eqref{eq:prob-cond1} and \eqref{eq:prob-cond2}.
\begin{thm} \label{thm:prob}
\added{Let $N\geq 1$ denoting the species pool size, and $k\leq N$ a given value of number of coexisting species, be two fixed strictly positive integers. Let $n$ be the random variable defined as the number of coexisting species, under random $\lambda_i$. Then the probability $P(n=k|N)$} \deleted{Probability} of the event "exact $k$-species-equilibrium in a random pool of size $N$" is:
\begin{equation} \label{eq:prob-simple}
\mathbb{P}(n = k|N)
=
\binom{N}{k} 
\frac{k (k-1)^{N-k}}{N}
\int_{[0,1]^{k-1}} 
\mathbbm{1}_{\left\{ \sum_{j=1}^{k-1} u_j < 1 \right\}}
\frac{\left(1 - \max_j u_j\right)^N}{\left(k - \sum_{j=1}^{k-1} u_j\right)^{N-k}}
\prod_{j=1}^{k-1} \frac{du_j}{(1-u_j)^2}.
\end{equation}
\end{thm}
%
%
%
%
%
%
\begin{proof}
$\bullet$ Firstly, by symmetry \replaced{of indexes }{ and condition $\lambda_k= \min \left\{\lambda_1,\dots,\lambda_k\right\}$}, the probability of the target event is that
\begin{equation} \label{eq:prob-form-1st}
\mathbb{P}(n = k\rvert N) 
= 
\binom{N}{k}k \,
\mathbb{P}\!\left(
\underbrace{\lambda_{k} \sum_{j=1}^{k-1} \frac{1}{\lambda_j} > k-2}_{\text{\eqref{eq:prob-cond1} on these $k$ species}}
\;\;\text{and}\;\;
\underbrace{\max_{j>k} \lambda_j < \frac{k-1}{\sum_{j=1}^k \frac{1}{\lambda_j}}}_{\text{\eqref{eq:prob-cond2} for the rest $N-k$ species}}
\right)\,,
\end{equation}
\added{in which, the probability on the right-hand-side has imposed upon, the condition $\lambda_k= \min \left\{\lambda_1,\dots,\lambda_k\right\}$.}

Let $t:= \min \left\{\lambda_1,\dots,\lambda_k\right\}=\lambda_k$. Conditional on $t$, the remaining $k-1$ variables are i.i.d. $Y_j \sim \mathrm{U}(t,1)$. 
\\ \\
$\bullet$ Secondly, we consider condition \eqref{eq:prob-cond1}, which \replaced{reads }{ is read as}
\begin{equation} \label{eq:new-cond1-prob}
\lambda_{k} \sum_{j=1}^{k-1} \frac{1}{\lambda_j} > k-2
\Longleftrightarrow
t\sum_{j=1}^{k-1}\frac{1}{Y_j} > k-2.
\end{equation}
$\bullet$ Thirdly, we consider condition \eqref{eq:prob-cond2}, which \replaced{reads }{ is read as}
\begin{equation} \label{eq:new-cond2-prob}
\max_{i>k} \lambda_{i} \sum_{j=1}^{k} \frac{1}{\lambda_j} < k-1 
\;
\Longleftrightarrow
\;
Y_i < \frac{k-1}{\sum_{j=1}^{k}\frac{1}{Y_j}},\;\forall i > k
\;
\Longleftrightarrow
\;
Y_i < \frac{k-1}{\frac{1}{t} + \sum_{j=1}^{k-1}\frac{1}{Y_j}},\;\forall k+1 \leq i \leq N.
\end{equation}
$\bullet$ Finally, noting that $t \in \left[0,1\right]$, from \eqref{eq:prob-form-1st}, \eqref{eq:new-cond1-prob} and \eqref{eq:new-cond2-prob}, the probability of the $k$-species-equilibrium is
\begin{equation*}
\mathbb{P}(n = k\rvert N) 
= \binom{N}{k} \, k \int_0^1 \int_{(t,1]^{k-1}} 
\mathbbm{1}_{\left\{ t \sum_{j=1}^{k-1} \frac{1}{y_j} > k-2 \right\}}
\left( \frac{k-1}{\tfrac{1}{t} + \sum_{j=1}^{k-1} \tfrac{1}{y_j}} \right)^{N-k}
dy \, dt\,.
\end{equation*}
By changing variables $r_j := t/y_j \in \left(t,1\right)$, due to $y_j \in \left(t,1\right)$ for all $j$, then changing the order of integrals, we have that
\begin{equation*}
\begin{aligned}
\mathbb{P}(n = k\rvert N) 
&
=
\binom{N}{k} \, k (k-1)^{N-k} 
\int_0^1 t^{N-1} \int_{[0,1]^{k-1}}
\mathbbm{1}_{\left\{t < \min r_j\right\}}
\mathbbm{1}_{\left\{ \sum_{j=1}^{k-1} r_j > k-2 \right\}}
\frac{1}{\bigl(1 + \sum_{j=1}^{k-1} r_j \bigr)^{N-k}}
\prod_{j=1}^{k-1} \frac{dr_j}{r_j^2} \, dt
\\
&
=
\binom{N}{k} \, \frac{k (k-1)^{N-k}}{N}
\int_{[0,1]^{k-1}} 
\mathbbm{1}_{\left\{ \sum_{j=1}^{k-1} r_j > k-2 \right\}}
\frac{\bigl(\min_j r_j \bigr)^{N}}{\bigl(1 + \sum_{i=1}^{k-1} r_i \bigr)^{N-k}}
\prod_{j=1}^{k-1} \frac{dr_j}{r_j^2}.
\end{aligned}
\end{equation*}
By changing variable $u_j:= 1 - r_j \in \left(0,1\right)$ for $1\leq j\leq k$, we obtain that
\begin{equation*}
\mathbb{P}\left(n=k\rvert N\right)
= 
\binom{N}{k} 
\frac{k (k-1)^{N-k}}{N}
\int_{[0,1]^{k-1}} 
\mathbbm{1}_{\left\{ \sum_{j=1}^{k-1} u_j < 1 \right\}}
\frac{\left(1 - \max_j u_j\right)^N}{\left(k - \sum_{j=1}^{k-1} u_j\right)^{N-k}}
\prod_{j=1}^{k-1} \frac{du_j}{(1-u_j)^2}.
\end{equation*}

\end{proof}

We note that, trivially, this Theorem \ref{thm:prob} is consistent with the Proposition \ref{prop:1unstable}.

\paragraph{Numerical computation of the probability distribution $P(n=k|N)$} To compute \eqref{eq:prob-simple}, we suggest that Monte Carlo (MC) methods are appropriate. In which, MC methods, or MC experiments, are a broad class of computational algorithms that rely on repeated random sampling to obtain numerical results.
 
We suggest MC methods because the integrand is high dimensional, non-smooth (indicator and $\max$), and has boundary singularities $\big((1-u_j)^{-2},\dots, (k-\sum u_j)^{-(N-k)}\big)$. These features break the smoothness assumptions behind tensor-product/adaptive quadrature and sparse grids (cost blows up with $k$ and accuracy deteriorates near corners), whereas MC’s $O(M^{-1/2})$ error is dimension- and smoothness-agnostic.

To evaluate the $(k-1)$-dimensional integral—with an indicator, a $\max$, and boundary singularities—
deterministic cubature scales poorly, so we use (quasi-)MC with importance sampling.
MC estimates the integral by the sample mean of the integrand at $M$ random/low-discrepancy 
points from the domain, giving an unbiased estimator with root-mean-square error $O(M^{-1/2})$.
We use Sobol points and a Beta-biased proposal that concentrates mass near $u_j \approx 1$ 
to match $(1-u_j)^{-2}$, which reduces variance; standard errors come from the sample variance.
\\ \\
We perform model simulations and highlight very well-matching results in Figure \ref{fig:10strains_probs}, in which the probability of the event "exact $k$-species-equilibrium" is computed in two ways: 
\begin{enumerate}
\item based on the long run behavior of generated invader-driven replicator dynamics, and
\item our probability formula in Theorem \ref{thm:prob}. 
\end{enumerate}
\begin{figure}[htb!]
\centering
\includegraphics[width=0.6\linewidth]{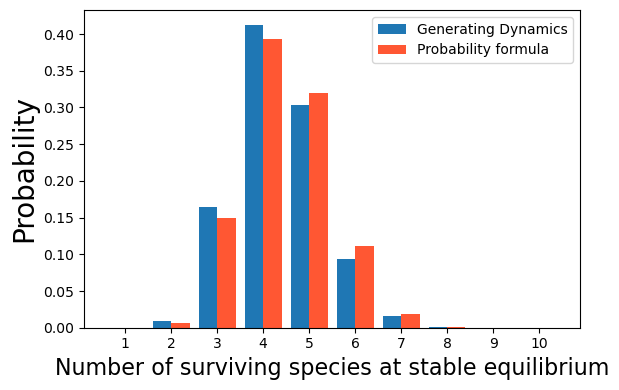}
\caption{\textbf{Probability of $n-$ species coexistence in a pool of $N$ species (here $N=10$)}. We plot the probability mass function of the number of coexisting species at the unique stable equilibrium of the invader-driven replicator \eqref{eq:repli_invader} with $N = 10$ and i.i.d. invader fitnesses
	$\lambda_i \sim \mathcal{U}[0,1]$.
	Blue bars: empirical frequencies from 10,000 independent ODE runs started in the simplex interior (survival counted when $z_i(T) > 10^{-4}$).
	Orange bars: probabilities obtained by Monte-Carlo evaluation of the Theorem \ref{thm:prob} integral formula, 
	reported with 95$\%$ confidence intervals (as listed in the Supplementary material S2). The two estimates yield similar probabilities of the target events. For a similar figure for $N=20$ see Figure \ref{fig:20strains_probs}.}
\label{fig:10strains_probs}
\end{figure}

\begin{rmk}
According to \eqref{eq:prob-cond1} and \eqref{eq:prob-cond2}, we can compute the probability using the following integral
\begin{equation} \label{eq:prob-complex}
\mathbb{P}(n = k\rvert N) 
= 
N! \int_{1 > u_1 > \cdots > u_N > 0} \mathbbm{1}_{\left\{u_{k+1} < \frac{k-1}{\sum_{j=1}^k \frac{1}{u_j}} < u_k \right\}} du_1 \cdots du_N.
\end{equation}
We observe that, formulation \eqref{eq:prob-complex} is more intuitive, meanwhile the expression \eqref{eq:prob-simple} is more easily computable, because of:
\begin{itemize}
\item Dimensionality: the first is $N$-dimensional; the second is only $(k-1)$-dimensional (independent of $N$).
\item {Domain:} \eqref{eq:prob-complex} integrates over an ordered simplex (awkward constraints); the other equation \eqref{eq:prob-simple} lives on the hypercube $[0,1]^{k-1}$ with a single simple constraint $\sum u_j < 1$.
\item {Structure:} \eqref{eq:prob-simple} has a smooth, factorable weight $(1 - \max u)^N / (k - \sum u)^{N-k} \prod (1-u_j)^{-2}$; easy for quadrature/Monte Carlo.
\item {Efficiency:} in \eqref{eq:prob-complex}, rejection/ordering strongly reduces Monte-Carlo efficiency as $N$ grows; the \eqref{eq:prob-simple} has no ordering/rejection.
\item {Asymptotics:} in the \eqref{eq:prob-simple}, $N$ appears only as exponents—handy for large $N$ analysis.
\end{itemize}
However, we remark that for small combinations of species pool size and number of coexisting species $(N,k)$ (e.g., $N=k=2$) the formulation \eqref{eq:prob-complex} is fine; otherwise the other integral formulation \eqref{eq:prob-simple} is more convenient to compute $\mathbb{P}(n=k\rvert N)$.

\end{rmk}

\subsection{Computing special low-dimensional cases of $\mathbb{P}(n=k|N)$}

Here we illustrate how we can compute the probability that $k$ species survive in an invader-driven replicator system of $N$ species, $\mathbb{P}(n=k|N)$, for small pool size, using the results on the selection mechanism of surviving species. Let us consider $X_1,X_2,\dots,X_N$ as $N$ independent and identically distributed random variables drawn from $\mathcal{U}\,[0,1]$, with realizations the invasiveness traits of our species, $\lambda_1,\lambda_2,\dots,\lambda_N$, respectively. Without loss of generality, we can order and renumber the invasion fitnesses such that $0<\lambda_N<\dots<\lambda_2<\lambda_1<1$ and treat them as realizations of the order statistics $\min\{X_1,X_2,\dots,X_N\}=:X_{(N)}<\dots<X_{(2)}<X_{(1)}:=\max\{X_1,X_2,\dots,X_N\}$. Moreover, rescaling the invasion fitnesses such that $0<\lambda_N<\dots<\lambda_2<\lambda_1=1$ we end up with $N-1$ order statistics, $X_{(2)},\dots,X_{(N)}$. The joint probability density function (PDF) of the $N-1$ order statistics is given by the expression,
\begin{align}
    f_{X_{(2)},\dots,X_{(N)}}(\lambda_2,\dots,\lambda_N) &= (N-1)!\,f(\lambda_2)\cdots f(\lambda_N) = (N-1)! \ ,
\end{align}
where we use that the PDF of $\mathcal{U}\,[0,1]$ is $f(\lambda_i) = 1$, $i=2,\dots,N$. Below we show how to compute these probabilities for low dimension.

\paragraph{Case $N=3$.}  Let us start with the easiest case, $N=3$ and $n=2$, with $0<\lambda_3<\lambda_2<\lambda_1=1$ and $f_{X_{(2)},X_{(3)}}(\lambda_2,\lambda_3) = 2!=2$. From \eqref{eq:prob-cond1}-\eqref{eq:prob-cond2}, we know that $k=2$ if $\lambda_3\leq Q_2^*=\frac{\lambda_1\,\lambda_2}{\lambda_1+\lambda_2} = \frac{\lambda_2}{1+\lambda_2}$, hence we compute $\mathbb{P}(n=2|N=3)$ integrating the PDF over the whole space \footnote{Note that the integration order has to be from $\lambda_N$ to $\lambda_2$. Recall that $\lambda_3\leq Q_2^*$ and $\lambda_3\leq Q_3^*$ are equivalent conditions, but the integral is computable only if we use $Q_2^*$.} using the indicator function $\mathbbm{1}_{\bigl\{ \lambda_3\leq Q_2^* \bigr\}}$,

\begin{align}
    \mathbb{P}(n=2|\,N=3) & 
    = 
    \mathbb{P}\left( \lambda_3\leq Q_2^* \right) = \int_{0<\lambda_3<\lambda_2<1} f_{X_{(2)},X_{(3)}}\, \mathbbm{1}_{\bigl\{ \lambda_3\leq Q_2^* \bigr\}}\,d\lambda_3 d\lambda_2 \\
    & = 2\int_0^{1} \int_0^{Q_2^*} d\lambda_3 d\lambda_2 = 2 \int_0^{1} Q_2^*\, d\lambda_2 = 2 \int_0^{1} \frac{\lambda_2}{1+\lambda_2}\, d\lambda_2 \nonumber= 2-2\,\ln(2) \approx 0.6137 \ . \nonumber
\end{align}
Then, the probability that all three species survive, in a pool size of $N=3$ is given by
\begin{align}
    \mathbb{P}(n=3|N=3) = 1-\mathbb{P}(n=2|N=3) = -1+2\ln(2) \approx 0.3863\ .
\end{align} 

\paragraph{Case $N=4$.} We proceed in a similar way as before with $N=4$, $n=2$, and $0<\lambda_4<\lambda_3<\lambda_2<\lambda_1=1$. Using the joint PDF $f_{X_{(2)},X_{(3)}, X_{(4)}}(\lambda_2,\lambda_3, \lambda_4) = 3!=6$ and imposing again that $\lambda_3\leq Q_2^*$, we obtain
\begin{align}
\mathbb{P}(n=2|N=4) & 
=
\mathbb{P}\left( \lambda_3\leq Q_2^* \right) 
= 
\int_{0<\lambda_4<\lambda_3<\lambda_2<1} f_{X_{(2)},X_{(3)}, X_{(4)}}\, \mathbbm{1}_{\bigl\{ \lambda_3\leq Q_2^* \bigr\}}\,d\lambda_4 d\lambda_3 d\lambda_2 \\
    & = 6\int_0^1\int_0^{Q_2^*}\int_0^{\lambda_3}d\lambda_4 d\lambda_3 d\lambda_2 = 6\int_0^1\int_0^{Q_2^*}\lambda_3\,d\lambda_3 d\lambda_2 = 3 \int_0^1 (Q_2^*)^2 \, d\lambda_2 \nonumber \\
    & = 3 \int_0^1 \frac{\lambda_2^2}{(1+\lambda_2)^2}\,d\lambda_2 = 3\,\left( \frac{3}{2} - 2\ln(2) \right)= \frac{9}{2} - 6\, \ln(2) \approx 0.3411\ .\nonumber
\end{align}
Now, for $N=4$ and $n=3$ the conditions are $\lambda_3>Q_2^*$ and $\lambda_4\leq Q_3^* = \frac{2}{1+\frac{1}{\lambda_2}+\frac{1}{\lambda_3}}$. Hence, 
\begin{align}
\mathbb{P}(n=3|N=4) & 
= 
\mathbb{P}\left( \lambda_3 > Q_2^*\  \bigcap \lambda_4\leq Q_3^* \right) \\
    & = \int_{0<\lambda_4<\lambda_3<\lambda_2<1} f_{X_{(2)},X_{(3)}, X_{(4)}}\, \mathbbm{1}_{\bigl\{ \lambda_3 > Q_2^*\  \bigcap \lambda_4\leq Q_3^*  \bigr\}}\,d\lambda_4 d\lambda_3 d\lambda_2 \nonumber\\
    & = 6\int_0^1\int_{Q_2^*}^{\lambda_2}\int_0^{Q_3^*} d\lambda_4\,d\lambda_3\,d\lambda_2 = 6\int_0^1\int_{Q_2^*}^{\lambda_2}Q_3^*\,d\lambda_3\,d\lambda_2 \nonumber\\
    &= \int_0^1 Q_2^*\int_{Q_2^*}^{\lambda_2}\frac{\lambda_3}{Q_2^* +\lambda_3}\,d\lambda_3\,d\lambda_2 \approx 0.5490 \ , \nonumber
\end{align}
which we solved with an algebraic calculator as computations get huge and complicated as the dimension increases. We present the results in Table \ref{tab:probabilities} for $N=3, 4, 5$, where the expected value of $n$ is also shown. 
The consistency of these results with those obtained from simulations provides further evidence in favour of the selection mechanism of coexisting species, uncovered in Theorems 5 and 9.

\begin{table}[t!]
\centering
\begin{tabular}{c|c|c|c|c|c|c}
\hline
\multirow{2}{*}{$\boldsymbol{N}$} & 
\multirow{2}{*}{$\boldsymbol{k}$}
 & \multicolumn{3}{c|}{\textbf{Analytical computations}} 
 & \multicolumn{2}{c}{\textbf{Simulations}} \\
\cline{3-7}
 & & \boldmath$\mathbb{P}(n=k|N)$  & \boldmath$\%$ & \boldmath$\mathbb{E}(n)$ & \boldmath$\%$ & \boldmath$\mathbb{E}(n)$ \\
\hline
\hline
\multirow{2}{*}{$3$} 
 & $2$ & $2-2\,\ln(2)$     & $61.37$ & \multirow{2}{*}{$2.386$} & $60.20$ & \multirow{2}{*}{$2.398$} \\
 & $3$ & $-1 + 2\,\ln(2)$  & $38.63$ &                          & $39.80$ &  \\
\hline
\multirow{5}{*}{$4$} 
 & $2$ & $\tfrac{9}{2}-6\,\ln(2)$ & $34.11$ & \multirow{5}{*}{$2.769$} & $35.6  0$ & \multirow{5}{*}{$2.749$} \\
 & \multirow{2}{*}{$3$} & $-2\,\pi^2 + 42\,\ln(2)-12-24\,\ln^2(2)$ & \multirow{2}{*}{$54.90$} & & \multirow{2}{*}{$53.90$} &  \\
 &                      & $-18\,\ln(3)-24\,Li_2(\added{-2} \cancel{3})$               &                           & & & \\
 & \multirow{2}{*}{$4$} & $2\,\pi^2 -36\,\ln(2)+\tfrac{17}{2}+24\,\ln^2(2)$ & \multirow{2}{*}{$10.99$} & & \multirow{2}{*}{$10.50$} &  \\
 &                      & $+18\,\ln(3)+24\,Li_2(\added{-2} \cancel{3})$               &                           & & & \\
\hline
\multirow{6}{*}{$5$} 
 & $2$ & $\tfrac{17}{2} - 12\,\ln(2)$ & $18.22$ & \multirow{5}{*}{$3.113$} & $18.20$ & \multirow{5}{*}{$3.097$} \\
 & \multirow{2}{*}{$3$} & $-105-24\,\pi^2 +420\,\ln(2)-288\,\ln^2(2)$ & \multirow{2}{*}{$54.69$} & & \multirow{2}{*}{$56.20$} & \\
 &                      & $-204\,\ln(3)-288\,Li_2(\added{-2} \cancel{3})$                  &                           & & & \\
 & $4$ & -- & $24.67$ & & $23.30$ &  \\
 & $5$ & -- & $2.33$ & & $2.30$ &   \\
\hline
\end{tabular}

\caption{\textbf{Explicit probabilities that $k$ species persist in an invader-driven system of $N$ species.} Here $\lambda_i\sim\mathcal{U}\,[0,1]$, and $E[n]$ denotes the expected value of $n$. Values are computed analytically from the expressions in Section 3.1, and numerically from model simulations  (1000 invader-driven systems for each $N$ and evaluation of the coexisting species selection mechanism in Theorem 5). For
$\mathbb{P}(n=4|N=5)$ and $\mathbb{P}(n=5|N=5)$ there are no closed expressions. Note that we obtain dilogarithms, defined as $
    Li_2(x) = \int_1    ^x\frac{\ln(1-u)}{u}du$, $x\in \mathbb{C}$. \added{The integral is taken on $\mathbb{C}$-plane, starting from $1$ and ending at $x$, and path-independence.}}
\label{tab:probabilities}
\end{table}

\subsection{Mean number of coexisting species: stable community size} \label{subsec:mean-NoSpecies}
Although we have \replaced{a representation of the}{ the explicit} probability distribution (Theorem \ref{thm:prob}), it is very hard to obtain a closed expression analytically for the mean number of coexisting species $E[n]$ for stable equilibria in the invader-driven case. 
\begin{figure}[h!]
    \centering
    \includegraphics[width=0.4\linewidth]{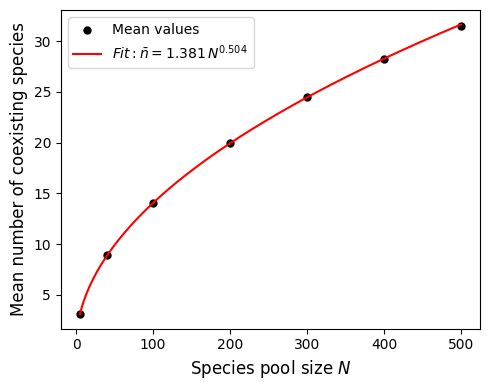}
    \caption{\textbf{Mean number of coexisting species vs. species pool size, in an invader-driven replicator dynamics.} We plot the mean number of coexisting species in the stable equilibrium $\mathbb{E}[n]$ against species pool size $N$, in invader-driven systems with positive invasion fitness traits. We superimpose the power law relation (red line) between $\mathbb{E}[n]$ and $N$, which has been fitted in log-log scale using values $N=5,\, 40,\, 100,\, 200,\, 300,\, 400,\, 500$. Each data point (black circle) is made out of $1000$ different systems with $\lambda_i \sim \mathcal{U}\, [0,1]$ $\forall\,i\in S.$ The plot shows the mean stable community size increases as $\sim \sqrt{2N}$.}
    \label{fig:nbar}
\end{figure}
In Figure \ref{fig:nbar}, summarizing several random realizations of the invader-driven $\Lambda$ matrix, we evaluate numerically how the mean number of coexisting species in a system increases with species pool size $N$. We observe that the relationship seems to be \replaced{well }{ perfectly}-captured by $$\mathbb{E}[n]\sim1.4\sqrt{N},$$ hence indicating a much less than linear increase of effective species number at equilibrium with pool size.
\begin{change}
\begin{thm}\label{th:En}
Let $\lambda_1\ge\lambda_2\ge\cdots\ge\lambda_N>0$ be the descending order statistics of i.i.d.
$\mathcal U[0,1]$ invasion fitness. Define $Q^*_k= 
\frac{k-1}{\sum_{j=1}^k \frac{1}{\lambda_j}}$ and  the random variable 
$n\in\{1,\dots,N\}$ of surviving species define as the unique integer such that 
$$\lambda_{n+1}\leq Q^*_n\leq \lambda_{n}.$$
 Then, as $N\to\infty$, we have  $$\dfrac{n}{\sqrt{2N}}\xrightarrow{\mathbb{P}}1$$  and in particular
$\mathbb{E}[n]=\sqrt{2N}+o\!\left(\sqrt N\right)$.

\end{thm}

\deleted{Indeed, we prove that $\mathbb{E}[n]$ can be approximated by $\sqrt{2N} + o(1)$. The proofs of lemmas in this section are presented in Section \ref{app:proofs}}. 
\\
The proof is done by a sequences of lemmas whose proofs are presented in the Supplementary Material.
\end{change}
Firstly, 
given random \added{invasion fitness traits among $N$ species} $\lambda_1, \ldots, \lambda_N \overset{\text{i.i.d.}}{\sim} \mathcal{U}[0,1]$, we focus on the \deleted{conditions involving the} distribution of \textit{ordered} values $\lambda_j$, \added{where $\lambda_j$ denotes} \deleted{denoting} the $j$-th largest sample value. We have the following Lemma \ref{lmm:PDF-Beta}, which is a well-known result and presented clearly in \cite[p.63]{Gentle2009ComputationalS} and \cite{Jones2009KumaraswamysDA}.

\begin{lmm} \label{lmm:PDF-Beta}
The probability density function (PDF) of ordered $j-th$ element $\lambda_j$, for all $1\leq j\leq N$, follows the Beta distribution \added{with parameters $(N+1-j,j)$} and is given by
$$
f_{\lambda_j}(x) 
= \frac{x^{N-j}(1-x)^{j-1}}{B(N+1-j,j)}
, 
\qquad 0 < x < 1,
$$
where
$$
B(N+1-j,j) = 
\frac{\Gamma(j)\Gamma(N+1-j)}{\Gamma(N+1)}
=
\frac{(j-1)!(N-j)!}{N!}.
$$
\end{lmm}
\noindent
\begin{change}
According to Theorem \ref{thm:main-converge}, we recall that, the unique equilibrium has a given number of $n$ species if and only if
\begin{equation} \label{eq:expect-condition}
\lambda_{n+1} < Q^*_n < \lambda_n
\end{equation}
with $\lambda_N \le \cdots \leq \lambda_1$ the order statistics.  
Thus we have the following equivalence:
\begin{equation} \label{eq:probequiv}
P(n=k|N)=P(\lambda_{k+1} < Q_k^* < \lambda_k),
\end{equation}
which maps in a way the probability of the number of coexisting species $n$ to the probability distribution of the ordered random $\lambda_j$. Namely we seek the average index $k$ among the highest $\lambda_i$ that satisfies \eqref{eq:expect-condition}.

However, the direct computation of $\mathbb{E}[n]$ stays hard to compute directly. But for large $N$, we are able to use \eqref{eq:expect-condition} thanks to a well-known concentration phenomenon of Beta variables. 

\begin{lmm}\label{lmm:concentrated}
Fix $C>0$. For each integer $N\ge 1$ and each $1\le k\le C\sqrt N$, let $\lambda_j$ denote the $j$-th largest order statistic of $N$ i.i.d.\ $\mathcal U[0,1]$ variables. Then for every $1\le j\le k$,
\[
  \mathbb E[\lambda_j] = 1 - \frac{j}{N+1}:= u_j.
\]
Moreover, as $N\to+\infty$, we have the concentration estimates (see Definition \ref{def:Op} for the definition of $\mathcal{O}_{\mathbb{P}}).$
\[
\lambda_j =  u_j + \mathcal O_{\mathbb P}(N^{-3/4})
\qquad\text{and}\qquad
Q^*_k =  S_k + \mathcal O_{\mathbb P}(N^{-3/4}),
\]
where
\[
 S_k := \frac{k-1}{\sum_{j=1}^k  \frac{1}{u_j}}
 \quad
 \text{and}
 \quad
 Q^*_k := \frac{k-1}{\sum_{j=1}^k \frac{1}{\lambda_j}} .
\]

\end{lmm}
\begin{proof}
    See Supplementary Materials Section \ref{app:proofs}.
\end{proof}
The next lemma shows that the deterministic quantities $\widehat{u_j}$ satisfy an asymptotic similar to the Theorem \eqref{th:En}. 
\begin{lmm}\label{lmm:determinEn}
Let $N\geq 2$ be an integer and define for each $k\in\{1,\cdots,N\}$, ${u}_k=1-\frac{k}{N+1}$ and ${S}_k= \frac{k-1}{\sum_{j=1}^k \frac{1}{u_j}}$.
 Then :
 \begin{enumerate}
 	\item[(i)] There exists a unique integer $\widehat{n}:=\widehat{n}(N)\leq N$ such that $${u}_{\widehat{n}+1}\leq {S}_{\widehat{n}}\leq {u}_{\widehat{n}}.$$
 	\item[(ii)] As $N\to +\infty$ we have the asymptotic profile $\widehat{n}\sim\sqrt{2N}.$ 
 \end{enumerate}
\end{lmm}
\begin{proof}
    See Supplementary Section \ref{app:proofs}.
\end{proof}
The previous asymptotic Lemma together with the concentration phenomena of Lemma \ref{lmm:concentrated}  yields the proof of Theorem \ref{th:En}.
\begin{proof}[Proof of Theorem \ref{th:En}]

   Let $\lambda_1 \ge \cdots \ge \lambda_N$ be the order statistics of i.i.d.\ $\mathcal U[0,1]$ variables.
By Proposition \ref{rmk:prop-k}, there exists a unique $n\in\{1,\dots,N\}$ such that
\[
\lambda_{n+1} \le Q^*_n \le \lambda_n.
\]
\medskip
\noindent
We reuse the notations $u_k$ and $S_k$ for all $1\leq k\leq N$ in Lemma \ref{lmm:concentrated} and Lemma \ref{lmm:determinEn}.
By Lemma~\ref{lmm:determinEn}, there exists a unique integer $\widehat n$ such that
\[
u_{\widehat n+1} \le S_{\widehat n} \le u_{\widehat n},
\qquad
\widehat n \sim \sqrt{2N}.
\]
\medskip
\noindent
By Lemma~\ref{lmm:concentrated}, for every $j=\mathcal O(\sqrt N)$,
\[
\lambda_j = u_j + \mathcal O_{\mathbb P}(N^{-3/4}),
\qquad
Q^*_{\widehat n} = S_{\widehat n} + \mathcal O_{\mathbb P}(N^{-3/4}).
\]
Thus
\[
\lambda_{\widehat n+1} \le Q^*_{\widehat n} \le \lambda_{\widehat n}
\]
with probability tending to $1$ as $N\to\infty$.
By uniqueness of $n$, this implies
\(
\mathbb P(n=\widehat n)\xrightarrow[N\to\infty]{}1.
\)
which reads also $n = \widehat n + o_{\mathbb P}(1).$
Together with $\widehat{n}\sim \sqrt{2N}$ this implies the final relation
\[
\frac{n}{\sqrt{2N}}
= \frac{\widehat n}{\sqrt{2N}} + o_{\mathbb P}(1)
\xrightarrow{\mathbb P} 1.
\]
\medskip
\noindent
Last, since $0\le n\le N$, the sequence $(n/\sqrt{2N})$ is uniformly integrable. Hence convergence in probability implies convergence in $L^1$, and therefore

\[
\frac{\mathbb E[n]}{\sqrt{2N}} \to 1,
\qquad
\mathbb E[n] = \sqrt{2N} + o(\sqrt N).
\]
\end{proof}

\end{change}

\noindent

\section{Sequential system assembly, niche saturation and mean invasion fitness $Q^*$}
Let us consider an invader-driven replicator system with all positive invasion fitnesses, which reaches its stable state (its global attractor), including $k$ species persisting. Assume that $\lambda_i \in \left[0,1\right]$ for all $i$ and $\lambda_k \leq \dots \lambda_2 \leq \lambda_1$. What happens when a new invader appears? 

We have several possibilities for the outcome of such event, which depends on both the invasion trait of the invader, and on the species and system traits of the resident community \citep{gjini2023towards, kurkjian2021impact, shea2002community, kolar2001progress, case1990invasion}. With our replicator framework, and for this particular case of invader-driven invasion fitness matrix, all these outcomes can be fully analytically predicted, and they include: 
\begin{itemize}
\item[i)]  \textit{community augmentation, }
\item[ii)] \textit{rejection failure,} and 
\item[iii)] \textit{species removal and replacement,}
\end{itemize}
with several sub-cases. 
We detail them below and provide a succinct and visual summary in Table \ref{tab:summary}.
\subsection{Conditions for new invasion outcomes, given $k$ coexisting species}\label{sec:sub:condition-invader}
\paragraph{\textbf{1. Community augmentation}} There are two particular cases for community augmentation: i) the invader is worse than the worst species, or  ii) the invader is better than at least one species. 
\\ \\
\textbf{(i)} \underline{Augmentation via adding a less fit species.} Now we find the necessary and sufficient condition for a new species $k+1$ which is less fit than these $k$ species, i.e. $\lambda_{k+1} \leq \lambda_i$, for all $1\leq i\leq k$, such that this specie $k+1$ can invade and persist in the dynamics at its stable state.

By the threshold in Theorem \ref{thm:main-converge}, the equivalent condition is that $Q^*_{k+1} < \lambda_{k+1}$. Recalling that $Q^*_{k+1} = \frac{k}{\sum_{j=1}^k \frac{1}{\lambda_k} + \frac{1}{\lambda_{k+1}}}$, by direct computation, we obtain that the equivalent condition $\lambda_{k+1} > \frac{k-1}{\sum_{j=1}^k \frac{1}{\lambda_j}}=Q^*_k$.
\\
In this case, the equivalent condition is as follows
\begin{equation}\label{eq:aug-c1}
Q^*_k<\lambda_{k+1} < \lambda_k 
\end{equation}

Since $\lambda_{k+1} > Q^*_{k} \Longleftrightarrow \lambda_{k+1} > Q^*_{k+1}$, then if we choose $\lambda_{k+1} \in \left(Q_{k}^*,\lambda_k\right)$, this leads to $\left(Q^*_{k+1},\lambda_{k+1}\right) \neq \emptyset$. This allows us to choose $\lambda_{k+2} \in \left(Q^*_{k+1},\lambda_{k+1}\right)$, yielding species $k+2$ can invade the $k+1$-species system. By induction, we can construct a sequence of species such that we can add them as many times as we want to invade the system.
\\ \\
\textbf{(ii)} \underline{Augmentation via adding a fitter than at least one species.} Now we find the necessary and sufficient condition for the community augmentation when the new species $k+1$ is fitter than at least one species. In this case, we obtain that $\lambda_k \leq \min\left\{\lambda_{k+1},\lambda_1,\dots,\lambda_{k-1}\right\}$, i.e. $\lambda_{k+1} > \lambda_k$.

By the threshold in Theorem \ref{thm:main-converge}, the equivalent condition is that $Q^*_{k+1} < \lambda_{k}$. By straightforward computations, we have 
\begin{equation*}
Q^*_{k+1} < \lambda_{k}
\Longleftrightarrow
\lambda_{k+1} 
<
\left( \frac{k}{\lambda_k} -\sum_{j=1}^{k}\frac{1}{\lambda_j} \right)^{-1}.
\end{equation*}
For the availability of $\lambda_{k+1}$, we have the condition
\begin{equation*}
\lambda_k < \left( \frac{k}{\lambda_k} -\sum_{j=1}^{k}\frac{1}{\lambda_j} \right)^{-1}
\Longleftrightarrow
\lambda_k > Q^*_{k},
\end{equation*}
which holds true by our assumption of $k$-species equilibrium.
Therefore, in this case, the equivalent condition is as follows. 
\begin{equation} \label{eq:aug-c2}
\lambda_k
<
\lambda_{k+1} 
<
\left( \frac{k}{\lambda_k} -\sum_{j=1}^{k}\frac{1}{\lambda_j} \right)^{-1}.
\end{equation}
\paragraph{\textbf{2. Rejection failure}} \underline{The new invader is rejected by the system.}  The invader goes to extinction but does not change the resident system. First, we note that for the purpose that the resident system is unchanged, $\lambda_{k+1} < \lambda_k$. Since, if $\lambda_{k+1} > \lambda_k$, the species $k+1$ goes extinct, leading to the extinction of species $k$ as well, by Theorem \ref{thm:main-converge}.

By the threshold in Theorem \ref{thm:main-converge}, we can prove that, the condition for the invader not to invade the resident system with $k$ species is that $\lambda_{k+1}<Q^*_{k}$. Indeed, if $\lambda_k > \lambda_{k+1} > Q^*_{k}$, this leads to the community augmentation.
\\ \\
It cannot happen that the new species invades the available persistent system, removes old species then itself goes extinct as well. Indeed, assume that new species $k+1$ invades and removes species $k'+1$ to $k$, then it goes extinct. This implies that at the long run behavior, there exists $i < k'+1 \leq k$ such that the set of persistent species is $\{1,\dots,i\}$. We deduce that $Q_i^* > \max\{\lambda_{i+1},\lambda_{k+1}\}$, which is absurd, since $k$ is the first index satisfying $Q^*_k > \lambda_{k+1}$ by Proposition \ref{rmk:prop-k}.
\\

Hence, the equivalent condition for the rejection failure is that $\lambda_{k+1}<Q^*_{k}$, hence the invader invasion fitness is lower than the mean invasion fitness of the resident system.

\paragraph{\textbf{3. Species removal and replacement}} The invader succeeds at the expense of one or more species in the system. Depending on where $\lambda_{k+1}$ falls in relation to the different ranges,
there is a different number of species to replace.
There may be two cases for new species $k+1$ replacing old species from $k' + 1$ to $k$ with $k' + 1 \leq k$: i) the least fit species of the new system is the recent invader, or ii) the least-fit species of the new system is the old $k'$-ranked species.

Note that, by Theorem \ref{thm:main-converge}, since species $k+1$ persists and species $k'+1$ goes extinct, then $\lambda_{k+1} > \lambda_{k'+1}$.
\\ \\
\textbf{(i)} \underline{The invader becomes the new least fit species of the new system.} By Theorem \ref{thm:main-converge}, the equivalence condition for new species $k+1$ replacing old species from $k' + 1$ to $k$ with $k' + 1 \leq k$ 
is that:
\begin{equation*}
\lambda_{k'} >
\lambda_{k+1}
>
Q^{*,\text{new}}_{k'+1} 
>
\lambda_{k' + 1},
\quad
\text{where }\;
Q^{*,\text{new}}_{k'+1} 
=
\frac{k'}{\sum_{j=1}^{k'} \frac{1}{\lambda_j} + \frac{1}{\lambda_{k+1}}}
\,.
\end{equation*}
By direct computations, we have that
\begin{equation*}
\lambda_{k+1}
>
Q^{*,\text{new}}_{k'+1} 
\Longleftrightarrow
\lambda_{k+1}
>
\frac{k'}{\sum_{j=1}^{k'} \frac{1}{\lambda_j} + \frac{1}{\lambda_{k+1}}}
\Longleftrightarrow
\lambda_{k+ 1} > Q^*_{k'}
\end{equation*}
and
\begin{equation} \label{eq:replace-bigger-condition}
\begin{aligned}
Q^{*,\text{new}}_{k'+1} 
>
\lambda_{k' + 1}
\Longleftrightarrow
\frac{k'}{\sum_{j=1}^{k'} \frac{1}{\lambda_j} + \frac{1}{\lambda_{k+1}}} 
>
\lambda_{k' + 1}
\Longleftrightarrow
\lambda_{k+1} 
> 
\left( \frac{k'}{\lambda_{k' + 1}} - \frac{k'-1}{Q^*_{k'}}  \right)^{-1}.
\end{aligned}
\end{equation}
On the other hand, by direct computations, we also get that
\begin{equation*}
Q^*_{k'} < \left( \frac{k'}{\lambda_{k' + 1}} - \frac{k'-1}{Q^*_{k'}}  \right)^{-1}
\Longleftrightarrow
\frac{Q^*_{k'}}{\lambda_{k' + 1}} < 1,
\end{equation*}
which holds true since $Q^*_i < \lambda_{i+1}$ for all $i < k$. This implies that
\begin{equation*}
\lambda_{k+1} 
> 
\left( \frac{k'}{\lambda_{k' + 1}} - \frac{k'-1}{Q^*_{k'}}  \right)^{-1}
\Longrightarrow
\lambda_{k+ 1} > Q^*_{k'}.
\end{equation*}
For the possible choice of a new species with $\lambda_{k+1}$, we need the equivalent condition that
\begin{equation*}
\lambda_{k'} > \left( \frac{k'}{\lambda_{k' + 1}} - \frac{k'-1}{Q^*_{k'}}  \right)^{-1}
\Longleftrightarrow
\frac{k'}{\lambda_{k'+1}}
>
\sum_{j=1}^{k'}\frac{1}{\lambda_j} + \frac{1}{\lambda_{k'}}.
\end{equation*}
Therefore, in this case, the equivalence condition for new species $k+1$ replacing old species from $k' + 1$ to $k$ with $k' + 1 \leq k$ and $\lambda_{k+1} < \lambda_{k'}$ is that
\begin{equation*}
\frac{k'}{\lambda_{k'+1}}
>
\sum_{j=1}^{k'}\frac{1}{\lambda_j} + \frac{1}{\lambda_{k'}}
\quad
\text{and }\;
\lambda_{k+1} 
> 
\left( \frac{k'}{\lambda_{k' + 1}} - \frac{k'-1}{Q^*_{k'}}  \right)^{-1}.
\end{equation*}
\textbf{(ii)} \underline{The least fit species of new system is the old species $k'$.} By Theorem \ref{thm:main-converge}, the equivalence condition for new species $k+1$ replacing old species from $k' + 1$ to $k$ with $k' + 1 \leq k$ 
is that:
\begin{equation*}
\lambda_{k+1}
>
\lambda_{k'} >
Q^{*,\text{new}}_{k'+1} 
>
\lambda_{k' + 1},
\quad
\text{where }\;
Q^{*,\text{new}}_{k'+1} 
=
\frac{k'}{\sum_{j=1}^{k'} \frac{1}{\lambda_j} + \frac{1}{\lambda_{k+1}}}
\,.
\end{equation*}
We reuse the condition in \eqref{eq:replace-bigger-condition} and compute directly the other condition that
\begin{equation*}
\begin{aligned}
Q^{*,\text{new}}_{k'+1} 
<
\lambda_{k'}
\Longleftrightarrow
\frac{k'}{\sum_{j=1}^{k'} \frac{1}{\lambda_j} + \frac{1}{\lambda_{k+1}}} 
<
\lambda_{k'}
\Longleftrightarrow
\lambda_{k+1} 
< 
\left( \frac{k'}{\lambda_{k'}} - \frac{k'-1}{Q^*_{k'}}  \right)^{-1}.
\end{aligned}
\end{equation*}
Thus, in this case, the condition for new species $\lambda_{k+1}$ is read as
\begin{equation*}
\max\left\{\lambda_{k'}, \left( \frac{k'}{\lambda_{k' + 1}} - \frac{k'-1}{Q^*_{k'}}  \right)^{-1}\right\}
<
\lambda_{k+1}
<
\left( \frac{k'}{\lambda_{k'}} - \frac{k'-1}{Q^*_{k'}}  \right)^{-1}
.
\end{equation*}
For the possible choice of a new species with $\lambda_{k+1}$, we need the equivalent condition that
\begin{equation*}
\left( \frac{k'}{\lambda_{k'}} - \frac{k'-1}{Q^*_{k'}}  \right)^{-1} > \lambda_{k'}
\Longleftrightarrow
\lambda_{k'} > Q^*_{k'-1},
\end{equation*}
which always holds because $k$ is the first index satisfying $Q^*_{k} > \lambda_{k+1}$ by Proposition \ref{rmk:prop-k}, and
\begin{equation*}
\left( \frac{k'}{\lambda_{k'}} - \frac{k'-1}{Q^*_{k'}}  \right)^{-1}
>
\left( \frac{k'}{\lambda_{k' + 1}} - \frac{k'-1}{Q^*_{k'}}  \right)^{-1}
\Longleftrightarrow
\lambda_{k'} > \lambda_{k' + 1},
\end{equation*}
which always holds as well.
\\ \\
To improve the understanding of these conditions in a concise manner, we provide their summary in Table \ref{tab:summary}, illustrating the exact mathematical principles dividing the various invasion regimes.

\begin{table}[htb!]
\centering
\setlength{\tabcolsep}{6pt}
\renewcommand{\arraystretch}{1.2}
\begin{tabular}{C{2.8cm}| m{4cm}| m{8.5cm}}
\toprule
\multicolumn{1}{c|}{\textbf{Invasion outcome}} & \multicolumn{1}{c|}{\textbf{Condition}} & \multicolumn{1}{c}{\textbf{Location of invader trait} }\\
\midrule
\multirow{2}{*}[-4ex]{\textbf{Augmentation}}
& $Q_k^*<\lambda_{k+1}<\lambda_k$
& \adjustbox{valign=c}{
\begin{tikzpicture}
  \draw[thick] (0,0) -- (8,0); 
  \foreach \x/\label in {0.0/\textbf{0}, 1/\textbf{1}} {
    \pgfmathsetmacro{\xpos}{8*\x}
    \draw[very thick] (\xpos,0.15) -- (\xpos,-0.15);
    \node[below] at (\xpos,-0.15) {\label};}

\foreach \x/\label in {
0.12/$\boldsymbol{Q_{k}^*}$,  0.28/$\boldsymbol{\lambda_{k}}$} 
  {
    \pgfmathsetmacro{\xpos}{10*\x}
    \draw[very thick, black] (\xpos,0.15) -- (\xpos,-0.15);
    \node[below, black] at (\xpos,-0.15) {\label};
  }
  
  
  \foreach \x/\label in 
  {0.2/$\boldsymbol{\lambda_{k+1}}$}
  {
    \pgfmathsetmacro{\xpos}{10*\x}
    \draw[very thick, red] (\xpos,0.15) -- (\xpos,-0.15);
    \node[above, red] at (\xpos,0.15) {\label};
}
  \node at (7,-0.5) {\textcolor{black}{\textbf{\dots}}}; 
; 
\end{tikzpicture}}
\\
\cmidrule(lr){2-3} 
& \makecell[l]{$\lambda_k<\lambda_{k+1}<U_k$,\\ $U_k:= \left(\frac{k}{\lambda_k} -\frac{k-1}{Q^*_k}\right)^{-1}$}
& 
\adjustbox{valign=c}{
\begin{tikzpicture}
  \draw[thick] (0,0) -- (8,0); 
  \foreach \x/\label in {0.0/\textbf{0}, 1/\textbf{1}} {
    \pgfmathsetmacro{\xpos}{8*\x}
    \draw[very thick] (\xpos,0.15) -- (\xpos,-0.15);
    \node[below] at (\xpos,-0.15) {\label};}

\foreach \x/\label in {
0.12/$\boldsymbol{Q_{k}^*}$,  0.28/$\boldsymbol{\lambda_{k}}$,
0.42/$\boldsymbol{U_{k}}$} 
  {
    \pgfmathsetmacro{\xpos}{10*\x}
    \draw[very thick, black] (\xpos,0.15) -- (\xpos,-0.15);
    \node[below, black] at (\xpos,-0.15) {\label};
  }
  
  
  \foreach \x/\label in 
  {0.35/$\boldsymbol{\lambda_{k+1}}$}
  {
    \pgfmathsetmacro{\xpos}{10*\x}
    \draw[very thick, red] (\xpos,0.15) -- (\xpos,-0.15);
    \node[above, red] at (\xpos,0.15) {\label};
}
  \node at (7,-0.5) {\textcolor{black}{\textbf{\dots}}}; 
; 
\end{tikzpicture}}
\\
\midrule
%
%
%
%
\multirow{1}{*}[-1.5ex]{\textbf{Rejection}}
& $\lambda_{k+1}\le Q_k^*$
& 
\adjustbox{valign=c}{
\begin{tikzpicture}
  \draw[thick] (0,0) -- (8,0); 
  \foreach \x/\label in {0.0/\textbf{0}, 1/\textbf{1}} {
    \pgfmathsetmacro{\xpos}{8*\x}
    \draw[very thick] (\xpos,0.15) -- (\xpos,-0.15);
    \node[below] at (\xpos,-0.15) {\label};}

\foreach \x/\label in {
0.12/$\boldsymbol{Q_{k}^*}$,  0.28/$\boldsymbol{\lambda_{k}}$} 
  {
    \pgfmathsetmacro{\xpos}{10*\x}
    \draw[very thick, black] (\xpos,0.15) -- (\xpos,-0.15);
    \node[below, black] at (\xpos,-0.15) {\label};
  }
  
  
  \foreach \x/\label in 
  {0.05/$\boldsymbol{\lambda_{k+1}}$}
  {
    \pgfmathsetmacro{\xpos}{10*\x}
    \draw[very thick, red] (\xpos,0.15) -- (\xpos,-0.15);
    \node[above, red] at (\xpos,0.15) {\label};
}
  \node at (7,-0.5) {\textcolor{black}{\textbf{\dots}}}; 
; 
\end{tikzpicture}}
%
\\
\midrule
\multirow{2}{*}[-4ex]{\makecell[l]{\textbf{Replacement} \\ (from $k'+1$ to $k$)}}
& 
\makecell[l]
{$\frac{k'}{\lambda_{k'+1}}
>
\frac{k'-1}{Q^*_{k'}} + \frac{1}{\lambda_{k'}}$,
\\
$ V_{k'} < \lambda_{k+1} 
<\lambda_{k'}$ \\
$V_{k'}:=
\frac{1}{\frac{k'}{\lambda_{k' + 1}} - \frac{k'-1}{Q^*_{k'}}} $}
& 
\adjustbox{valign=c}{
\begin{tikzpicture}
  \draw[thick] (0,0) -- (8,0); 
  \foreach \x/\label in {0.0/\textbf{0}, 1/\textbf{1}} {
    \pgfmathsetmacro{\xpos}{8*\x}
    \draw[very thick] (\xpos,0.15) -- (\xpos,-0.15);
    \node[below] at (\xpos,-0.15) {\label};}

\foreach \x/\label in {
0.28/$\boldsymbol{\lambda_{k}}$,
0.35/$\boldsymbol{V_{k'}}$,
0.45/$\boldsymbol{\lambda_{k'}}$} 
  {
    \pgfmathsetmacro{\xpos}{10*\x}
    \draw[very thick, black] (\xpos,0.15) -- (\xpos,-0.15);
    \node[below, black] at (\xpos,-0.15) {\label};
  }
  
  
  \foreach \x/\label in 
  {0.42/$\boldsymbol{\lambda_{k+1}}$}
  {
    \pgfmathsetmacro{\xpos}{10*\x}
    \draw[very thick, red] (\xpos,0.15) -- (\xpos,-0.15);
    \node[above, red] at (\xpos,0.15) {\label};
}
  \node at (7,-0.5) {\textcolor{black}{\textbf{\dots}}}; 
; 
\end{tikzpicture}}
\\
\cmidrule(lr){2-3}
&
\makecell[l]{
$W_{k'} < \lambda_{k+1}
<
U_{k'}$
\\
$W_{k'} :=
\max\left\{\lambda_{k'},  V_{k'} \right\}
$}
& 
\adjustbox{valign=c}{
\begin{tikzpicture}
  \draw[thick] (0,0) -- (8,0); 
  \foreach \x/\label in {0.0/\textbf{0}, 1/\textbf{1}} {
    \pgfmathsetmacro{\xpos}{8*\x}
    \draw[very thick] (\xpos,0.15) -- (\xpos,-0.15);
    \node[below] at (\xpos,-0.15) {\label};}

\foreach \x/\label in {
0.28/$\boldsymbol{\lambda_{k}}$,
0.35/$\boldsymbol{V_{k'}}$,
0.45/$\boldsymbol{\lambda_{k'}}$,
0.6/$\boldsymbol{U_{k'}}$ } 
  {
    \pgfmathsetmacro{\xpos}{10*\x}
    \draw[very thick, black] (\xpos,0.15) -- (\xpos,-0.15);
    \node[below, black] at (\xpos,-0.15) {\label};
  }
  
  
  \foreach \x/\label in 
  {0.52/$\boldsymbol{\lambda_{k+1}}$}
  {
    \pgfmathsetmacro{\xpos}{10*\x}
    \draw[very thick, red] (\xpos,0.15) -- (\xpos,-0.15);
    \node[above, red] at (\xpos,0.15) {\label};
}
  \node at (7,-0.5) {\textcolor{black}{\textbf{\dots}}}; 
; 
\end{tikzpicture}}
\\
\bottomrule
\end{tabular}
\caption{\textbf{Conditions for the outcome of an invader species $k+1$ in a system with $k$ coexisting species.} Visual summary of Section \ref{sec:sub:condition-invader}. 
Red ticks mark the trait of the invader $\lambda_{k+1}$ on the $[0,1]$ axis relative to the residents $\{\lambda_j\}_{j \leq k}$ 
and thresholds: mean invasion fitness of the resident system $Q_k^* = \frac{k-1}{\sum_{j=1}^k \lambda_j^{-1}}$ (left invasion threshold), $U_k = \left(\frac{k}{\lambda_k} - \sum_{j=1}^k \lambda_j^{-1}\right)^{-1}$ (upper augmentation threshold), and for replacement with a focal resident $k' \leq k$,
$V_{k'} = \left(\frac{k}{\lambda_{k'}} - \frac{k-1}{Q_{k'}^*}\right)^{-1}$,  $W_{k'} = \max\{\lambda_{k'}, V_{k'}\}$.
Each row indicates one of 3 main outcomes, whether: i) the system is {augmented} with the invader, ii) the invader is rejected, or iii) the invader {replaces} one or more residents 
according to the interval in which $\lambda_{k+1}$ falls.
}
 \label{tab:summary}
\end{table}

\subsection{Niche saturation via subsequent invasions of less fit species}
Here we focus on conditions allowing for community augmentation via subsequent invasions of less fit species (case 1,i in Section 4.1). New invasions in this case become more and more difficult. Random-trait invaders become harder and harder to be added to the system. Indeed, from \eqref{eq:aug-c1} and \eqref{eq:aug-c2},
we prove that 
$$\lim\left(\lambda_k - Q^*_k\right) 
=
0,
\quad
\text{as }\;
k \to \infty,
$$ 
indicating that the width of the permissible interval for invader traits, allowing successful invasion, becomes smaller as species system size increses.
First, we claim that $Q^*_{k+1} > Q^*_k$. Indeed, we have that
\begin{equation*}
\left\{
\begin{aligned}
&
Q^*_{k+1} &&= \frac{k}{\sum_{j=1}^{k+1}\frac{1}{\lambda_j}},\quad
z_i^{*,k+1} = 1-\frac{Q^*_{k+1}}{\lambda_i},\forall i=1,2,\dots,k+1
\,,\\
&
Q^*_k &&= \frac{k-1}{\sum_{j=1}^{k}\frac{1}{\lambda_j}},\quad
z_i^{*,k} = 1-\frac{Q^*_k}{\lambda_i},\forall i=1,2,\dots,k\,.
\end{aligned}
\right.
\end{equation*}
Thus, there exists $1\leq i\leq k$ such that $z_i^{*,k+1} < z_i^{*,k}$ i.e. $1-\frac{Q^*_{k+1}}{\lambda_i} < 1-\frac{Q^*_k}{\lambda_i}$, which implies $Q^*_k < Q^*_{k+1}$.
Hence, $\left(Q_{k+1}^*, \lambda_{k+1}\right) \subset \left(Q_k^*, \lambda_k\right)$.

Now, for all $k \geq 1$, let $A_k = \sum_{j=1}^k \frac{1}{\lambda_j}$ and $H_k = \frac{k}{A_k}$ be the harmonic mean. 
By monotone boundedness, $\lambda_k \downarrow L \ge 0$. By direct computation, we can check that
$$
\lim_{k \to \infty} \frac{k - (k-1)}{A_k - A_{k-1}}
= \lim_{k \to \infty} \frac{1}{\lambda_k^{-1}} = L.
$$
Since the sequence $\left\{A_k\right\}_{k\geq 1}$ is increasing to infinity, Stolz–Cesàro's theorem gives
$$\lim_{k \to \infty} H_k
=
\lim_{k \to \infty}\frac{k}{A_k}
= L.$$
But $Q^*_k = \frac{k-1}{A_k} = \frac{k-1}{k} H_k \to L$ as $k\to \infty$, hence $\lambda_k - Q^*_k \to L - L = 0$.


Roughly speaking, the invasion of a new weaker species into the available system becomes increasingly difficult as the number of coexisting species increases. This niche saturation phenomenon becomes the underlying reason for the saturating number of coexisting species with pool size $N$, seen in Figure \ref{fig:nbar}.

\section{Linking the invader-driven replicator to biological systems}
Until now, we have studied in detail key stability and coexistence principles in a replicator system characterized by invader-driven invasion fitness matrix between interacting species. Next, we link this type of replicator to biological contexts of relevance for modeling and empirical testing. 
\subsection{An equivalent Lotka-Volterra multispecies model}
According to \cite[Theorem 7.5.1]{Hofbauer_Sigmund_1998}, there exists a differentiable, invertible map, mapping the orbits of the replicator equation
$$
\frac{d}{dt}{z}_i = z_i \bigl( \left(\Lambda \mathbf{z}\right)_i - \mathbf{z}^T\Lambda \mathbf{z} \bigr),\qquad i=1,\ldots,N, 
$$
onto the orbits of the Lotka--Volterra equation
\begin{equation} \label{eq:LV}
\frac{d}{dt}{y}_i = y_i \left( r_i + \sum_{j=1}^{N-1} a_{ij} y_j \right), 
\qquad i=1,\ldots,N-1,    
\end{equation}
where $r_i = \lambda_i^N - \lambda_N^N$ and $a_{ij} = \lambda_i^j - \lambda_N^j$.

In the particular case, when we have the invader-driven replicator system with $\lambda_{i}^j = \lambda_i$ for all $1\leq j \leq N$, $j\neq i$ and $\lambda_i^i = 0$ for all $i$, we can obtain it as a transformation from the Lotka-Volterra system with the following definition:
\begin{equation*}
r_i = \lambda_i^N - \lambda_N^N = \lambda_i, 
\qquad 
a_{ij} = \lambda_i^j - \lambda_N^j =
\begin{cases}
-\lambda_N, & j=i, \\
\lambda_i - \lambda_N, & j \neq i,
\end{cases}
\qquad (i,j \leq N-1).
\end{equation*}
Concretely, the corresponding Lotka--Volterra system turns out to be of a special kind, and is explicitly written as:
\begin{equation} \label{eq:GLV-invader}
\frac{d}{dt}{y}_i 
= y_i \left( \lambda_i (1-y_i) + \left(\lambda_i - \lambda_N\right) \sum_{j=1}^{N} y_j \right),
\quad
i=1,\ldots,N-1.
\end{equation}
Without loss of generality we can take the $N-$th species to be the one with the largest $\lambda_i$, hence making the LV model an explicit competition model where the negative effect from the total abundance of all species is felt by species $i$ with coefficient $\lambda_i-\lambda_N.$
\\
Another remark is that \eqref{eq:GLV-invader} has a ``rank-one mean-field" interaction matrix:
$$
A = \underbrace{(\lambda_i - \lambda_N)_i \, \mathbf{1}^\top}_{\text{rank one}} - \lambda_N I,
$$
so species interact only through the total biomass 
$\sum_j y_j$. Thus all the results proven in this paper can be applied to such specific Lotka-Volterra model.

\subsection{\added{A model of competition of self-reproducing macromolecules}}
\added{
Beyond theoretical population dynamics, this class of replicator equations holds profound implications for biochemistry, particularly in modeling the origin of life and the \textit{in-vitro} evolution of self-replicating macromolecules \citep{paczko2024stochastic}. As demonstrated by \cite{epstein1979competitive} and \cite{hofbauer1981competition}, this mathematical framework accurately captures the selection kinetics of competing polynucleotides and catalytic macromolecular networks under constant organization constraints. Indeed, a generalized formulation of the competitive dynamics studied in these foundational works can be written as
$$\frac{dz_i}{dt } = z_i\left(G_i(z_i) - \frac{\phi}{c}\right), \quad 
\text{with } \phi = \sum_{j=1}^N z_j G_j(z_j) 
,
\quad
i=1,\dots,N \;\text{ and } c>0,$$
where the $G_i$ are strictly decreasing functions from $\mathbb{R}^+$ into $\mathbb{R}$ representing density-dependent growth rates. Under these conditions, the orbits trivially converge to the invariant concentration simplex
$$\Delta^c
=
\left\{
(x_1,\ldots,x_N)\in \mathbb{R}^N
\;:\;
x_i \ge 0,\;
\sum_{i=1}^N x_i = c,\;
i=1,\ldots,N
\right\}.$$
In aligning our theoretical baseline with this biochemical context, we acknowledge that the fundamental global attraction principle is a well-established result. Specifically, \cite{epstein1979competitive} established the kinetic boundaries dictating competitive exclusion versus coexistence, while \cite{hofbauer1981competition} provided rigorous proofs of global stability for such mass-action systems using the Lyapunov function mentioned in \cite[Exercise 19.3.5]{Hofbauer_Sigmund_1998}. In this article, the convergence of kinetically inferior replicators toward a global attractor (as detailed in Section \ref{sec:main-result-dynamics}) directly builds upon these classical models.
\\
However, rather than merely proving the existence of a global attractor, our present work advances this framework by explicitly constructing a deterministic algorithm to identify the exact state of global attraction for this specific invader-driven replicator equation. Additionally, we extend the classical analysis to include the case of non-positive invasion fitnesses, which is thoroughly examined in Appendix \ref{supp:negative}. Furthermore, leveraging our proposed algorithm and stability criteria, we analytically compute the expected number of asymptotically coexisting species in the long-term limit and mathematically characterize the dynamical outcome when a novel invader is introduced into an established system.}

\subsection{Multi-strain SIS system: trait assumptions for invader-driven replicator}
Next, we show how a replicator system, which is invader-driven, can be obtained from a real biological or epidemiological system bottom-up. Recall the SIS model with coinfection and multiple interacting species in \citep{madec2020predicting} and \citep{le2023quasi}. This model, under strain similarity, leads to replicator dynamics of strain frequencies circulating in the host population. The pairwise invasion fitness in such bottom-up derived replicator is an explicit function (weighted average) of trait asymmetries between species: 
\begin{equation}
\begin{aligned}
\lambda_i^j 
= 
\theta_1 
\left(b_i - b_j\right) 
+ \theta_2 
\left(-\nu_i + \nu_j\right) 
+ \theta_3 
\left(-u_{ij} - u_{ji} + 2u_{jj}\right) 
+ \theta_4 
\left(\omega_{ij}^i - \omega_{ji}^j\right) 
+ \theta_5 
\left(\mu \left(\alpha_{ji} - \alpha_{ij}\right) + \alpha_{ji} - \alpha_{jj}\right),
\end{aligned}
\end{equation}
for all $i,j$,
with conventions and notations of parameters defining species in our model are in \cite[Table 1]{le2023quasi}. We describe 3 straightforward cases of epidemiological trait assumptions for the multi-strain system, mapping to invader-driven replicator dynamics.
\begin{itemize}
\item[i)]Recalling Table 2 in \cite{madec2020predicting}, we can have an invader-driven invasion fitness matrix, when our system has variation in only co-colonization susceptibilities between species $k_{ij} = k +\epsilon\alpha_{ij}$. Indeed, if the co-colonization interaction matrix $K$ satisfies a special kind of symmetry: $$k_{ij}= \added{k_{ji}} \cancel{k_{ij}} = k_{ii} + k_{jj} - k, \quad \forall 1\leq i,j \leq N,$$
then the invasion fitness matrix is invader-driven with $\lambda_i^j = \alpha_{ii}$, for all $i,j$. 
\item[ii)] Another case can be constructed for a system with perturbation only in co-colonization interactions \cite{madec2020predicting}, yielding invader-driven replicator dynamics between species.
Indeed, let the target invader-driven row values be $\{\lambda_i\}_{1\leq i\leq N}$ so that $\lambda_{ij} = \lambda_i$ for $i \neq j$ (and $\lambda_{ii} = 0$ comes for free).
Take any small $\mu > 0$-which is the ratio of single to co-infection of the system, and pick any constant $c \in \mathbb{R}$. By defining the $\alpha$-matrix as follows:
$$
\alpha_{ij} :=
\begin{cases}
\frac{c - \lambda_i}{\mu}
, & i \neq j, \\
(\mu+1)\frac{c - \lambda_i}{\mu} - c, & i = j,
\end{cases}
$$
%
then, for $i \neq j$, we obtain the pairwise-invasion fitness
$$
\lambda_{i}^j 
= 
\mu(\alpha_{ji} - \alpha_{ij}) 
+ \alpha_{ji} 
- \alpha_{jj} 
= 
(\lambda_i - \lambda_j) 
+\frac{c - \lambda_j}{\mu}
- (\mu+1)\frac{c - \lambda_j}{\mu} + c
=  \lambda_i,
$$
and for $i = j$, $\lambda_{ii} 
= 0$ trivially.  
Hence $\Lambda$ is invader-driven with each row $i$ equal to $\lambda_i$ for all $1\leq i\leq N$.
\item[iii)] 
In the other epidemiological case when species vary only in clearance rates of co-infection \citep{le2023quasi}, we can also obtain an invader-driven invasion fitness matrix. Since coinfection includes same-strain and ordered mixed-strain combinations, there are four such traits defining each pair of species.
Indeed, we construct an example for co-infection clearance rates $\gamma_{ij}$'s, recalling from the similarity assumption that $\gamma_{ij} = \gamma +\epsilon u_{ij}$ for all $i,j$. 

Let the desired invader-driven row values be $\{\lambda_{i}\}_{i=1}^N$ with 
$\lambda_{i}^i = \lambda_i$ for $j \ne i$ and $\lambda_{i}^i = 0$ (for all $i$).
Pick any $c \in \mathbb{R}$ and define a symmetric $U = (u_{ij})_{1\leq i,j\leq N}$ such that:
$$
u_{ii} = c + \frac{\lambda_1 - \lambda_i}{2}, 
\qquad 
u_{ij} = u_{ji} = c + \frac{\lambda_1 - (\lambda_i + \lambda_j)}{2} \quad (i \ne j).
$$
Then, for $i \ne j$,
$$
\lambda_{i}^j = -u_{ij} - u_{ji} + 2u_{jj} 
= -2\left(c + \frac{\lambda_1 - \lambda_i - \lambda_j}{2}\right) 
+ 2\left(c + \frac{\lambda_1 - \lambda_j}{2}\right) = \lambda_i,
$$
and $\lambda_{i}^i = -2u_{ii} + 2u_{ii} = 0$.
Thus $\Lambda$ is invader-driven with rows $i$ constant equal to $\lambda_i$\footnote{Actually, we note that non-symmetric solutions also exist: choose any $u_{ij}$ and set 
$u_{ji} = 2u_{jj} - \lambda_i - u_{ij}$ for $i \ne j$ with $u_{ii}$ as above.}. 

All these examples show that an invader-driven replicator system be constituted bottom-up or can arise when modelling frequencies of co-circulating species with coinfection interactions. Their analysis can be made simpler by harnessing the analytical results on coexistence provided here.  
\end{itemize}

\subsection{Testing whether an ecological multispecies system is invader-driven}\label{sec:sub-test}

When considering a single ecological site and assuming a steady state scenario between all species present, by the uniqueness result and the analytically explicit form of the equilibrium $z_i$, we can always find a $\lambda_i$ $(i=1,\cdots,n)$ vector parameterization that satisfies the equilibrium expression (Eq.\ref{invader-driven}), and fits each species frequency as the equilibrium. However, a deeper and subtler way to test is by comparing different ecological sites where some of the same species are present.

In such context, to test the invader-driven replicator hypothesis implies interpreting species frequency data from multiple settings as coming from an equilibrium, and considering the sets they share in common, and checking whether they verify mathematical constraints (Eq.\ref{eq:test-linear}). We did this in an epidemiological context, namely in a \textit{Streptococcus pneumoniae} serotype colonization dataset \citep{dekaj2022pneumococcus,le2025inference}, where we have 5 geographic settings: Denmark \cite{Harboe2012pneumococcal}, Iran \cite{Tabatabaei2014multiplex}, Brazil \cite{rodrigues2017pneumococcal}, Nepal \cite{kandasamy2015multi}, and Mozambique \cite{adebanjo2018pneumococcal}. These settings report different number and identities of serotypes, as well as their frequencies. However they share seven serotypes in common, these being 14, 19A, 19F, 3, 4, 6A, 6B. Denoting their identities by $i \in {14,19A,19F,3,4,6B},$ we expect under the assumption of invader-driven $\Lambda$ underlying the serotype dynamics in such system, that all these seven serotypes should have the same invader-driven fitness $\lambda_i$ everywhere, independently of the context. 
\\

We tested for this hypothesis using the common serotypes, by considering pairwise the values of $(1-z_i)^{-1}$ in each setting, and under the equilibrium assumption. We took one country as a reference, without loss of generality, here Denmark, and examined whether there was a linear relationship between $\frac{1}{1-z_i}$ reported in Denmark and in the other countries. Under the invader-driven $\Lambda$ hypothesis, from equation \ref{invader-driven}, assuming equality of $\lambda_i$, we should expect a linear relationship to hold between observations in setting $s$ and $k$, for all common serotypes $i$: 

\begin{equation} \label{eq:test-linear}
\frac{1-z_i^{(\text{site k})}}{1-z_i^{(\text{site s})}}
=
\frac{Q^{(\text{site k})}}{Q^{(\text{site s})}}\,,
\end{equation}
with the slope of the relation being the inverse ratio of mean invasion fitness in each site at equilibrium, and the intercept of the regression expected to be zero. This theoretical property is illustrated in Figure \ref{fig:linear_example}, which continues the example in Figure \ref{fig:dynamics}.
\begin{figure}[htb!]
    \centering
    \includegraphics[width=0.38\linewidth]{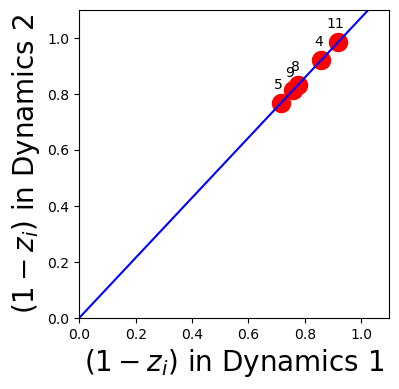}
    \caption{\added{\textbf{Expected linearity of $1-z_i$ equilibrium frequencies for species shared between distinct sites.} Respectively, the $x$- and $y$-axes represent $1 - z_i^*$ for all persistent species $i$ at equilibrium in Dynamics 1 (top panel) and in Dynamics 2 (bottom panel) in Figures \ref{fig:dynamics}c and \ref{fig:dynamics}f). Surviving species are highlighted in red and labeled by their respective identities. Species shared between the two systems (specifically, species 4, 5, 8, 9, and 11) exhibit the predicted linear relationship. 
    }}
    \label{fig:linear_example}
\end{figure}

The correlation results for the common serotypes shared among \textit{all} settings are summarized in Figure \ref{fig:7strains-test-linear}, and for serotypes shared two by two, in Figure \ref{fig:correlationPlots} taking another site as reference. From these data, we do not find unambiguous support for the linear relationships expected under the invader-driven $\Lambda$ hypothesis (see Supplementary Material S.2.2). Indeed, different more general invasion matrix alternatives were explored and fitted to this dataset in a previous study \citep{le2025inference}. 
\begin{figure}[htb!]
    \centering
    \includegraphics[width = 0.7 \textwidth]{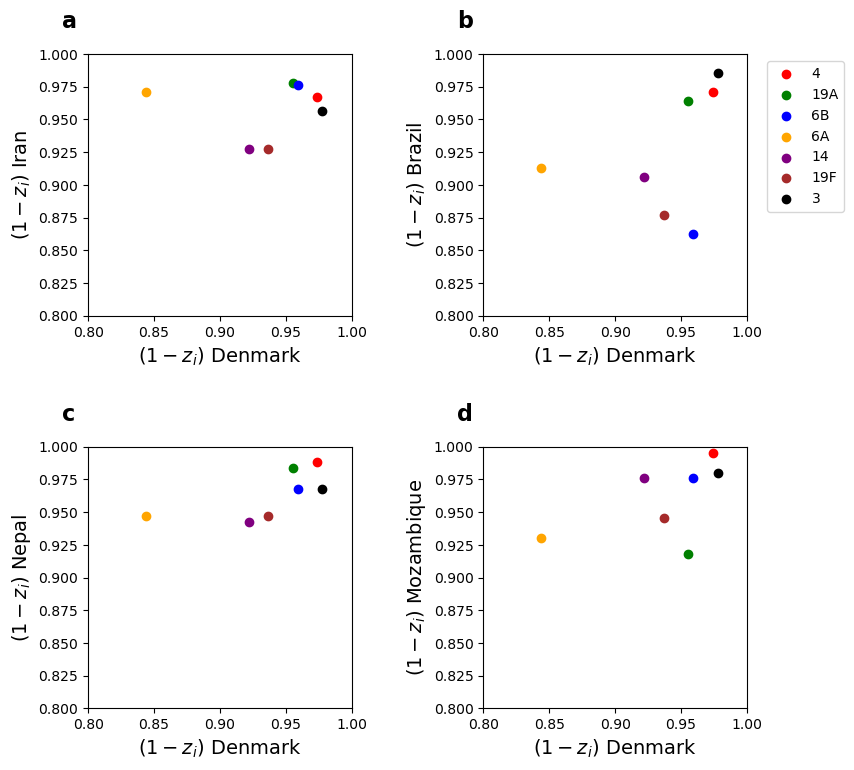}
    \caption{\textbf{Testing for invader-driven $\Lambda$ matrix in a multi-strain respiratory pathogen: \textit{Streptococcus pneumoniae}.} Common serotypes in pairs of countries/epidemiological settings (data analyzed in \citep{le2025inference}) are shown by the colored circles. 
    Here we take Denmark as the reference setting, and compare the reported $1-z_i$ in each other setting to the value reported in Denmark. If the invader-driven $\Lambda$ matrix hypothesis is correct, first there should be a significant linearity in each pair of observations, and secondly there should be a proportionality constant linking these slopes among settings. Contrary to this hypothesis, we find little correlation between $1-z_i$ values in these pairs of countries, the exact  Pearson correlations of $1 -z_i$, $i$ in the set of common species, being: $\text{corr}\left(\text{Denmark, Iran}\right) = 0.049$, $\text{corr}\left(\text{Denmark, Brazil}\right) = 0.366$, $\text{corr}\left(\text{Denmark, Nepal}\right) = 0.667$, $\text{corr}\left(\text{Denmark, Mozambique}\right) = 0.530$. To gain more statistical power, we explored separate sets of overlapping serotypes taking another setting as reference: Nepal (Figure\ref{fig:correlationPlots}).}
    \label{fig:7strains-test-linear}
\end{figure}

However, even if these linearity relationships for overlapping serotypes were to be satisfied, the invader-driven coexistence within replicator dynamics implies other constraints also for the magnitudes of these slopes and the intercepts (see Supplementary Material \ref{supp:test}), and potentially for the sets of site-specific serotypes. Thus, the linearity test should only be seen as a first requirement, warranting more quantitative verifications. 

Overall, we believe that the results exposed throughout in this paper provide a solid mathematical foundation to build upon a formal statistical framework to match real data from a wide range of ecological systems with the theoretical principles of this type of replicator equation.

\section{Discussion}

The replicator equation is central in evolutionary game theory with many applications in economics, biology, epidemiology and social sciences, and a long history of study and mathematical interest\citep{hofbauer1981competition, Hofbauer_Sigmund_1998, cressman2014replicator}. There are several special cases of replicators studied in the literature, including symmetric, zero-sum and other games \citep{diederich1989replicators,  yoshino2008rank,chawanya2002large, galla2006random}, and often with different methods, ranging from dynamical systems to statistical physics. Here, we have focused and studied in detail another special case of the replicator equation, which when expressed in terms of pairwise invasion fitnesses (zeros on the diagonal of the payoff matrix) has the structure of no variation along columns, but only along rows. This means that species are defined by their active invasiveness propensity in the system, and the result of each pairwise interaction depends only on the donor and not on the recipient species, that is why, we term this type of matrix \textit{invader-driven}. \added{Interestingly, this replicator dynamics has also been studied in classical early work on replicators with competitive coexistence of self-replicating macromolecules \citep{epstein1979competitive, hofbauer1981competition} but recent uses and applications of it have been more limited \citep{paczko2024stochastic}.} 
As studied briefly in our \replaced{recent}{previous} papers \citep{madec2020predicting,gjini2023towards}, this type of invader-driven matrix is prone to generate coexistence between multiple species, and is typically characterized by \added{very fast selection} and a quick rise of mean systemic invasion resistance when species get selected starting from equal initial frequencies; indeed the fastest rise when compared to other matrix structures such as the symmetric or random type, with entries in the same range \citep{gjini2023towards}. 

In this paper, we have gone deeper into all the stability, invasibility and coexistence regimes in this type of replicator system.
We have shown that when positive and negative fitness species co-occur initially, only a subset of the positive invaders may ultimately persist and coexist, and in a unique stable steady state. When the species pool consists entirely of negative invaders (all `losers'), alternative stable competitive exclusion states with a single species arise, and asymptotic dynamics depends on initial frequencies (see Supplementary material \ref{supp:negative} where we derive precise conditions). 

In the polymorphic steady state, - our main focus here - where only positive invaders coexist, equilibria are explicit, and species frequencies follow exactly the hierarchy of their invasiveness traits (Eq. \ref{invader-driven}). We have found a critical threshold that determines the number and identities of coexisting species $n$, starting with a random pool of $N$, - a threshold that is independent of any probability distribution where $\lambda_i$ generate from. Algorithmic evaluation of this criterion is very quick and allows for high-dimensional species pools to be efficiently analyzed. We have provided an expression for the probability of $k$ species coexistence in the random uniform distribution case $\mathbb{P}(n=k|N)$, and an approximation for the mean number of coexisting species $\mathbb{E}[n]$, increasing with pool size $\sim \sqrt{2}\sqrt{N}$. This result highlights the difficulty of obtaining large stable communities under this type of replicator, even when the species pool size is big.

\subsection*{Implications for invasion ecology}

Community ecology has a long-standing interest to predict biological invasions by applying modern niche concepts to invasive species and the communities they enter. The idea of ``niche opportunity" describes the conditions - such as resources, enemies, and environmental factors - that promote invasions, especially when these vary over time and space \citep{shea2002community}. While the efforts to integrate theories and reconcile conflicting data are ongoing for many years, e.g. on the role of species diversity, invader and resident traits, and more recently on interaction network properties \citep{case1990invasion,kolar2001progress,kurkjian2021impact,hui2019invade,hui2021trait}, formal and robust mathematical descriptions of such processes, enabling prediction, are scarce in the literature. 

Moreover, one of the outstanding questions in the field remains how can invasion performance and invasibility be explicitly incorporated into stability criteria for open, adaptive ecological networks with realistic topologies \citep{hui2019invade}, \added{and what are the rules underlying species packing in open communities \citep{chiwira2026species}.}

Here, we contribute to fill this gap, focusing on the invasion dynamics in a special kind of ecological system (invader-driven replicator dynamics). This is one topology of a mutual invasion network underlying interaction non-specificity on the receiving partner, that typically leads to a unique stable coexistence steady state. For this topology, we have characterized quantitatively in detail, the key balance between current mean invasion fitness of the system, and constituent species traits, and the next species (new invader) traits, allowing for different invasion outcomes. Our results highlight the importance and nonlinear effect of history and contingency in system assembly for these types of replicators, and confirm that new invasions, whenever they happen (Table \ref{tab:summary}), bringing the system to a new equilibrium, increase the mean invasion fitness of the system, thereby reducing invasion chances for future invaders.

Achieving the desired synergies or mitigation effects with biological invasions requires a clear understanding of which invasion outcomes occur and how they develop over time. While one requirement for this is reliable data on the status and trends of alien species \citep{seebens2025biological}, the other requirement is the existence of robust mathematical and quantitative frameworks. With the theoretical proofs provided here, we hope to contribute to a more comprehensive and robust assessment of invasion dynamics in a class of important ecological systems, and to directions for effective responses. We have also matched our analytic expressions with numerical examples and illustrations to show the validity of our approach and to improve understanding of simple low-dimensional cases.

\subsection*{\added{Limitations and} outlook}

\added{Although most of our results are independent of the specific distribution where the invasion growth rates of the species originate (Sections 2, 4, 5), some specific calculations such as probability of n-species coexistence were calculated for a particular distribution, namely the uniform (Section 3). Certainly in the future, it will be interesting to study other probability distributions and test how their properties will affect the distribution and statistics of the number of coexisting species for this type of replicator. } 

\added{Another direction that we did not pursue was to compute exactly the probabilities of certain invasion outcomes, whose criteria are derived in Section 4. This can be an interesting avenue for further extension and computation.} 

Finally, with an applications perspective in mind, we have pointed to several links of this type of special replicator with concrete biological systems. Indeed the replicator equation written in terms of pairwise invasion fitnesses (studied here) can be obtained under different models with the invasion matrix coefficients explicit functions of underlying trait asymmetries between the `player' species \citep{madec2020predicting,le2023quasi,gjini2023towards}. \added{Beyond the original old known application for self-replicating macromolecules \citep{epstein1979competitive,paczko2024stochastic},} here, in particular, we provided the equivalence links with rank-one Lotka-Volterra systems, and with a specific epidemiological multi-strain SIS scenario with co-colonization/co-infection \citep{madec2020predicting,le2023quasi}. 

In the latter scenario, we have shown that this type of invader-driven replicator dynamics, leading to multi-strain coexistence, requires epidemiologically (bottom-up) the presence of special \added{trait-mediated} interactions between strains, namely either via pairwise susceptibilities to co-colonization, or variability between strains in co-colonization durations, resonating with our previous results \citep{le2022disentangling}. This finding supports co-colonization/co-infection as a coexistence mechanism in multi-strain endemic systems and infectious disease epidemiology, and here advances the key principles for such coexistence within the invader-driven replicator dynamics. 

These examples highlight that the replicator equation with emergent net invasion fitness coefficients is well suited to examine the relative importance of underlying fundamental trait axes between species, condensing in a few parameters the key selective processes, and will enhance our ability to predict the outcome of invasions in diverse communities. The challenge remains translating mathematical results on the magnitudes of $\lambda$ coefficients to criteria and constraints on species raw life-history trait values and distributions. 

We have also applied the SIS-coinfection-replicator model to a currently important respiratory pathogen, namely to pneumococcus serotype frequencies from different settings \citep{dekaj2022pneumococcus,le2025inference}, and here we have illustrated with data how one could in principle test empirically for the invader-driven hypothesis in an ecological system. \added{As discussed in Section \ref{sec:sub-test} (see also \ref{supp:test}),} exposing the main expected axes of model-data convergence \added{under equilibrium assumptions (Eq. \ref{eq:test-linear})} sets the ground for more sophisticated statistical frameworks to be built upon these mathematical principles in the future. \added{Certainly the features of this replicator should inspire applications in other wide-ranging multispecies assembly and coexistence contexts, from epidemiology and ecology to biochemistry. }

\section*{Acknowledgement} We thank J. Tomás Lázaro for interesting discussions on several aspects of this work. The study received funding by the Portuguese Foundation for Science and Technology (FCT grant number 2022.03060.PTDC) and was partially supported by the European Union (Grant Agreement Number: 101080528— NOSEVAC). 

\section*{Data availability statement}
The data used for illustration in Section \ref{sec:sub-test}. can be found at this \href{https://github.com/lthminhthao/Invader-driven_ReplicatorSystem}{GitHub repository.}

\bibliographystyle{apalike}
\bibliography{article_refs}

\newpage

\begin{center}
\textbf{Supplementary material }\\
\title{Understanding the invader-driven replicator dynamics}
\end{center}
{Thi Minh Thao Le,}
{Marina Garcia-Romero, }
{João Duarte Alcântara Galvão, }
{Sten Madec, and}
{Erida Gjini}

\renewcommand\thefigure{S\arabic{figure}}\setcounter{figure}{0}
\renewcommand\thetable{S\arabic{table}}\setcounter{table}{0}
\renewcommand\thesection{S\arabic{section}}\setcounter{section}{0}
\renewcommand\thesubsection{S\arabic{subsection}}
\renewcommand{\theequation}{\thesection.\arabic{equation}}
\renewcommand{\thethm}{S\arabic{thm}}
\renewcommand{\theprop}{S\arabic{prop}}
\renewcommand{\thelmm}{S\arabic{lmm}}
\renewcommand{\thecor}{S\arabic{cor}}

\section{Negative invader-driven fitness case leads to competitive exclusion} \label{supp:negative}
In this section, we only consider a particular case of invader-driven fitness which has not considered in previous parts: all fitnesses are negative.

\begin{thm} \label{thm:neg-stable}
\added{
For any $N > 2$ and $(\lambda_1, \ldots, \lambda_N)$ with $\lambda_i < 0$ for all $i$, the dynamics \eqref{eq:repli_invader} converges to an equilibrium for every initial state. }
\begin{enumerate}
\item[i)] \added{Every monomorphic state is an asymptotically stable equilibrium.}
\item[ii)] \added{All non-monomorphic equilibria (including the unique interior equilibrium and face equilibria) are unstable.}
\end{enumerate}
	\deleted{For any $N \geq 1$ and $\left(\lambda_1,\lambda_2,\dots,\lambda_N\right)$ with $\lambda_{i} < 0$ for all $i \geq 1$,   there exists a global attractor of  \eqref{eq:repli_invader}, which is a boundary equilibrium}.
\end{thm}
\begin{proof}
\begin{change}
We prove that the invader driven dynamics converges to an equilibrium for any initial condition. Firstly, for any subset $S\subseteq\{1,\dots,N\}$, $|S|=m$ (i.e. $S$ has $m$ elements), we write
$F_S=\{z\in\Delta: z_i=0,\;\forall\, i\notin S\}$ for the
corresponding face, and $\mathrm{int}\,F_S$ for its relative interior.
\\
Now, similarly to the proof of Theorem \ref{thm:main-converge}, we take the potential 
\begin{equation*}
V\left(z\right)
:= 
\sum_{i=1}^{N} \left(\lambda_iz_i - \frac{\lambda_i}{2}z_i^2\right)
\quad
\Longrightarrow \,f_i\left(z\right):=\frac{\partial V}{\partial z_i} = \lambda_i\left(1-z_i\right)
\quad
\text{and} \quad
\dot V(z)
  = \sum_{i=1}^N z_i\bigl(f_i(z)-Q(z)\bigr)^2 \;\ge\; 0\,.
\end{equation*}
By \cite[Theorem 19.5.1]{Hofbauer_Sigmund_1998}, \eqref{eq:repli_invader} is a Shahshahani gradient on the simplex $\Delta=\{z\in \mathbb{R}^N:\, z_1+\dots+z_N = 1\}$ with the potential function $V$. This means that $V$ increases along every non-constant trajectory with equality if and only if
$f_i(z)=Q(z)$ for all $i$ with $z_i>0$,
i.e.\ if and only if $z$ is an equilibrium.

On each invariant set $F_S$, an interior equilibrium satisfies
$\lambda_i(1-z_i)=\nu_S$ for all $i\in S$, $\sum_{i\in S}z_i=1$.
This system has a unique solution, namely
$$
  z_i = 1-\frac{\nu_S}{\lambda_i},\qquad
  \nu_S = \frac{m-1}{\sum_{j\in S}1/\lambda_j},
$$
which lies in $\mathrm{int}\,F_S$ if and only if $0<z_i<1$ for all $i\in S$; otherwise $F_S$ has no interior equilibrium.
Since there are finitely many faces ($2^N$ subsets $S$) and at most one interior equilibrium per face, the total number of equilibria is finite leading to equilibria are isolated.

On the other hand, $\Delta$ is compact and $V$ is continuous and non-decreasing along any trajectory, hence $V(z(t))\to V^*$ for some~$V^*$.\end{change}

\added{By LaSalle's invariance principle, the $\omega$-limit set is contained in the largest invariant subset of $\{z\in \Delta:\;\dot V=0\}$, which equals the set of equilibria. Since equilibria are isolated, every trajectory converges to one equilibrium.}
\\ \\
\begin{change}
\textit{i)}
Without loss of generality, we denote the unique survivor by strain $1$, we have that $z_1^* = 1$, $Q^* = 0$, and $z_j^*=0$ for all $j \neq 1$. We recall the Jacobian at $e_1 = \left(1,0,\dots,0\right)$ as follows
\begin{equation*}
J(e_1) =
\begin{pmatrix}
	0 & -\lambda_2 & -\lambda_3 & \cdots & -\lambda_N \\
	0 & \lambda_2   & 0          & \cdots & 0 \\
	0 & 0           & \lambda_3  & \cdots & 0 \\
	\vdots & \vdots & \ddots     & \ddots & \vdots \\
	0 & 0 & 0 & \cdots & \lambda_N
\end{pmatrix}\,.
\end{equation*}
\deleted{which yields to our conclusion since $\lambda_j < 0$ for all $j \geq 1$.} \added{Note that, the dimension of the simplex $\Delta$ is $N-1$ then the spectrum of $J(e_1)$ restricted to $\Delta$ is all strictly negative since $\lambda_j < 0$ for all $j \geq 1$, which yields to our conclusion.}
\\ \\
\added{\textit{ii)} For this, we first claim that, if $z^*$ is a non-vertex equilibrium, then every neighborhood of $z^*$ in $\Delta$ contains at least one point $\bar{z}$ with $V(\bar{z}) > V(z^*)$.}

Indeed, since $z^*$ is not a vertex, it has at least two strictly positive components
$z_i^* > 0$ and $z_j^* > 0$ for some $i \neq j$. Consider $\bar{z} := \bar{z}(\delta)= z^* + \delta(e_i - e_j)$,
which stays in $\Delta$ for $|\delta|$ small enough.
The function $g(\delta) := V(\bar{z}(\delta))$ is strictly convex in $\delta$, since $\bar{z}(\delta)$ is affine in $\delta$ and $V$ is strictly convex on $\mathbb{R}^N$ (the Hessian $\nabla^2 V=-\mathrm{diag}(\lambda_1,\dots,\lambda_N)$ is positive definite because all $\lambda_i<0$).\\
Its derivative at $\delta=0$ is
$$
g'(0)
= \nabla V(z^*) \cdot (e_i - e_j)
= f_i(z^*) - f_j(z^*).
$$
Since $z^*$ is an equilibrium with $z_i^* > 0$ and $z_j^* > 0$, the equilibrium condition gives
$f_i(z^*) = Q(z^*) = f_j(z^*)$, so $g'(0)=0$.
A strictly convex function with zero derivative at $\delta=0$ has a strict minimum there, which implies
$$ V(\bar{z})=
V\bigl(z^* + \delta(e_i - e_j)\bigr)
= g(\delta) > g(0) = V(z^*)
\qquad \text{for all small } \delta \neq 0.
$$
For $|\delta|$ small enough, $\bar{z}$ lies in the given neighborhood of $z^*$, which proves the claim.
\\

\added{Now, we back to prove \textit{ii)}. Take any non-monomorphic equilibrium and choose $\epsilon > 0$ small enough that $B(z^*,\epsilon) \cap \Delta$ contains no other equilibrium, which is possible since equilibria are isolated. Take any $r > 0$. By strict convexity of $V$ and the fact that $z^*$ is not a vertex, $B(z^*,r) \cap \Delta$ contains a point $z_0$ with $V(z_0) > V(z^*)$. The trajectory from $z_0$ satisfies $V(z(t)) \geq V(z_0) > V(z^*)$ for all $t \geq 0$, since $\dot{V} \geq 0$. So $z(t)$ can never converge to $z^*$. By \textit{i)}, $z(t)$ converges to some other equilibrium, which lies outside $B(z^*,\epsilon)$, i.e. $z(t)$ eventually leaves $B(z^*,\epsilon)$.}
\end{change}

\end{proof}

\begin{prop} \label{prop:a.e.vertexconverge}
\added{
For all $\left(\lambda_1, \lambda_2,\dots,\lambda_N\right) \in (-\infty,0)^N$
outside a measure-zero exceptional set, almost every trajectory of \eqref{eq:repli_invader} in $\Delta$ converges to a monomorphic equilibrium.}
\end{prop}
\begin{change}
Initially, we need the following lemma.
\begin{lmm} \label{lmm:0-measure}
Let $(-\infty,0)^N$
be the parameter space of all-negative fitness vectors $\lambda=\left(\lambda_1, \lambda_2,\dots,\lambda_N\right)$. For a subset $S\subseteq \{1,\ldots,N\}$ with $|S|=m\geq 2$, define $\nu_S(\lambda)
=
\frac{m-1}{\sum_{j\in S}\frac{1}{\lambda_j}}$ and 
the exceptional set 
$$
E
=
\bigcup_{\substack{S\subseteq\{1,\ldots,N\}\\ |S|\geq 2}}
\;
\bigcup_{i\notin S}
E_{S,i},
\quad
\text{with }\;\;
E_{S,i}
=
\left\{\lambda\in (-\infty,0)^N:\;\lambda_i=\nu_S(\lambda)\right\}.
$$
Then $E$ has Lebesgue measure zero in $(-\infty,0)^N\subseteq\mathbb{R}^N$.    
\end{lmm}
\end{change}

\begin{proof}
\added{The index set of pairs $(S,i)$ with $|S|\geq 2$ and $i\notin S$ is finite since there are at most $2^N\cdot N$ such pairs, so it suffices to show each $E_{S,i}$ has measure zero.}

\added{Fix $S$ with $|S|=m\geq 2$ and $i\notin S$. On $(-\infty,0)^N$, the condition $\lambda_i=\nu_S(\lambda)$ is equivalent to
$\lambda_i\sum_{j\in S}\frac{1}{\lambda_j}
=
m-1$.
We define $g_{S,i}(\lambda)
=
\lambda_i\sum_{j\in S}\frac{1}{\lambda_j}
-
(m-1)$, which is a rational function, trivially not identically zero on $(-\infty,0)^N$, and we have that $E_{S,i}=\left\{\lambda\in (-\infty,0)^N:\;g_{S,i}(\lambda)=0\right\}$.}

\added{We recall the result, if $g:U\to\mathbb{R}$  is real-analytic on a connected open set $U\subseteq\mathbb{R}^N$
 and $g\not\equiv 0$,then  $\{g=0\}$ has Lebesgue measure zero. Note $(-\infty,0)^N$ is connected and open, so the lemma applies and we obtain $E_{S,i}=\{g_{S,i}=0\}$ has measure zero, leading to the conclusion of Lemma \ref{lmm:0-measure}.}

\end{proof}

\added{Now we go back to the proof of Proposition \ref{prop:a.e.vertexconverge}.}

\begin{proof}
\added{In this proof, we show that if $\left(\lambda_1, \lambda_2,\dots,\lambda_N\right) \in (-\infty,0)^N \setminus E$, in which $E$ is defined in Lemma \ref{lmm:0-measure}, the conclusion of Proposition \ref{prop:a.e.vertexconverge} holds.}

\added{Firstly, we note that $\Delta$ is an $(N-1)$-dimensional manifold with corners, embedded in the hyperplane
$H=\left\{ z \in \mathbb{R}^N:\;\sum_i z_i = 1
\right\}$.
The Lebesgue measure on $\Delta$ is the $(N-1)$-dimensional measure inherited from $H$. All dimensional and measure statements below are with respect to this $(N-1)$-dimensional structure.}

\begin{change}
Since every trajectory converges to some equilibrium, the set of equilibria is the disjoint union of vertices $\{e_1,\ldots,e_N\}$ and non-vertex equilibria $\{z^{(1)},\ldots,z^{(L)}\}$, where $L \leq 2^N - N$ is finite. So the $\Delta$ decomposes as
$$
\Delta
=
\left( \bigcup_{i=1}^N B_i \right)
\cup
\left( \bigcup_{\ell=1}^L C_\ell \right),
$$
where $B_i = \{z_0 \in \Delta : z(t;z_0) \to e_i\}$ is the basin of vertex $e_i$, and $C_\ell = \{z_0 \in \Delta : z(t;z_0) \to z^{(\ell)}\}$ is the basin of non-monomorphic equilibrium $z^{(\ell)}$. We show $\bigcup_\ell C_\ell$ has measure zero by showing that $\dim \bigcup_\ell C_\ell \leq N-2$. For this purpose, we apply the Stable Manifold Theorem, by computing the dimension of stable space of the linearization of the dynamics at any non-monomorphic equilibrium. 

We claim that, the linearization of the dynamics at  a non-monomorphic equilibrium $z^{(S)}\in \operatorname{int} F_S$  for some $S$ with $|S| = m \geq 2$, restricted to the $(N-1)$-dimensional $\Delta$,
\begin{enumerate}
\item[\textit{i)}] has $m-1 \geq 1$ strictly positive eigenvalues tangential\footnote{For a linear operator $A$ and an $A$-invariant subspace $W$, the tangential eigenvalues are the eigenvalues of the restriction operator $A|_W$, which govern the internal dynamics within $W$.
Conversely, the transverse eigenvalues are the eigenvalues of the induced quotient operator on the space $V/W$,
characterizing the stability and dynamics normal to the subspace.} to $F_S$, and,
\item[\textit{ii)}] has no zero eigenvalue.
\end{enumerate}
To prove \textit{i)}, let $J$ be the Jacobian of the restricted dynamics at $z^{(S)}$, acting on the  $(m-1)$-dimensional tangent space $T = \left\{\xi\in\mathbb{R}^S:\sum_{i\in S}\xi_i=0\right\}$.

Let $\eta\in T$ be an eigenvector with eigenvalue $\mu$, i.e. $ J\eta=\mu\eta$.
The Jacobian at equilibrium at $z^{(S)}$ (where $f_i\left(z^{(S)}\right) = Q\left(z^{(S)}\right)$ for all $i\in S$) is computed as
$$
J_{ij}\left(z^{(S)}\right)
=
z_i
\left(
\frac{\partial f_i}{\partial z_j}
-
f_j
-
\sum_{k=1}^N z_k
\frac{\partial f_k}{\partial z_j}
\right).
$$
By the Shahshahani metric, taking $\xi \in T$, we have that
$$
\langle \xi,J\eta\rangle_z
=
\sum_{i\in S}\frac{1}{z_i}\xi_i\eta_i
=
\sum_{\substack{i\in S, \\ 1\leq j \leq N}}
\left(
\frac{\partial f_i}{\partial z_j}
-
f_j
-
\sum_{k=1}^N z_k\frac{\partial f_k}{\partial z_j}
\right)
\xi_i\eta_j
=
\sum_{\substack{i\in S, \\ 1\leq j \leq N}}
\frac{\partial f_i}{\partial z_j}\,
\xi_i\eta_j,
$$
as $\sum_{i \in S} \xi_i = 0$.
Recalling that $ \frac{\partial f_i}{\partial z_j}
=
-\lambda_i\delta_{ij}
=
|\lambda_i|\delta_{ij}
$
and setting $\xi=\eta$ lead to
$$
\mu \|\eta\|_z^2
=
\langle \eta,J\eta\rangle_z
=
\sum_{i\in S} |\lambda_i|\,\eta_i^2
>
0
\Longrightarrow
\mu >0\,.
$$
Hence, every tangential eigenvalue is strictly positive, which leads to that the number of strictly positive eigenvalues is $m-1$. 
\end{change}
\\ \\
\added{To prove \textit{ii)}, we can compute explicitly, at a non-monomorphic equilibrium $z^{(S)}$ supported on $S$ (with $|S|=m\geq 2$), the transverse eigenvalues are
$\lambda_k-\nu_S$, for all $k\notin S$,
where $\nu_S=\frac{m-1}{\sum_{j\in S} 1/\lambda_j}$.
A zero eigenvalue occurs precisely when
$\lambda_i=\nu_S$ for some $i\notin S$, leading to $\left(\lambda_1,\lambda_2,\dots,\lambda_N\right) \in E$, which is absurd.}
\\ \\
\begin{change}
Since all of eigenvalues are nonzero, the basin $C_\ell$ equals the stable manifold, a smooth sub-manifold of dimension equal to the number of negative eigenvalues, hence $\dim C_\ell \leq N-m$,
by the Stable Manifold Theorem for smooth dynamical systems.
Hence, we deduce that
$$
\dim C_\ell \leq N-m \leq N-2 < N-1 = \dim \Delta,
$$
The basin of non-monomorphic equilibria is $\bigcup_{\ell=1}^{L} C_\ell$,
a finite union of measure-zero sets, hence has measure zero. Therefore, for almost every $z_0 \in \Delta$, the trajectory $z(t;z_0)$ converges to a monomorphic equilibrium.
\end{change}

\end{proof}

\subsection*{Dependence of winning species on initial conditions}
Now, we consider how initial conditions select the unique surviving species in the system. 
Firstly, we consider the particular case of 2-strain system and aim to find a criteria for initial value such that either strain is the unique survivor. \added{The proof of Theorem \ref{thm:neg-2strains} below is obvious so we skip it.}

\begin{thm} \label{thm:neg-2strains}
Given invader driven replicator system \eqref{eq:repli_invader} with $N=2$ and assume that $0 > \lambda_1 > \lambda_2$. Then there is a sharp threshold $$
z_1^{\dagger} := \frac{\lambda_1}{\lambda_1 + \lambda_2} \in (0,1)
$$
such that
$$
z_1(0)
\begin{cases}
	< z_1^{\dagger} \;\; \Rightarrow \;\; z(t) \to e_2 = (0,1), \\
	= z_1^{\dagger} \;\; \Rightarrow \;\; z(t) \equiv \left(z_1^{\dagger},\, 1 - z_1^{\dagger}\right), \\
	> z_1^{\dagger} \;\; \Rightarrow \;\; z(t) \to e_1 = (1,0),
\end{cases}
\; \text{ as }\; t \to \infty\,.
$$
\end{thm}
\deleted{\textit{Proof.} When $N=2$, $z_2 = 1 - z_1$ and \eqref{eq:repli_invader} now reads}
$$
\cancel{\frac{d}{dt}{z}_1 = z_1 \left( \lambda_1 (1 - z_1) - Q \right), 
\qquad Q = (\lambda_1 + \lambda_2) z_1 (1 - z_1),}
$$
\deleted{hence becomes}
\deleted{\begin{equation} \label{eq:N2-neg-eq}
\cancel{\frac{d}{dt}{z}_1 = z_1 (1 - z_1) \left(\lambda_1 - (\lambda_1 + \lambda_2) z_1 \right).} 
\end{equation}}
\deleted{From \eqref{eq:N2-neg-eq}, $z_1(1-z_1) > 0$ on $(0,1)$, so 
$\operatorname{sign}(\frac{d}{dt}{z}_1) = \operatorname{sign}\left(\lambda_1 - (\lambda_1 + \lambda_2) z_1 \right)$.  
Hence}
\deleted{$$\cancel
{\frac{d}{dt}{z}_1 =
\begin{cases}
	< 0, & z_1 \in (0, z_1^{\dagger}), \\
	= 0, & z_1 = z_1^{\dagger}, \\
	> 0, & z_1 \in (z_1^{\dagger}, 1).
\end{cases}}
$$}
\deleted{Therefore any solution with $z_1(0) > z_1^{\dagger}$ is monotone increasing, bounded above by $1$, 
so $z_1(t) \to L \in [z_1^{\dagger}, 1]$. If $L < 1$ then by continuity $\frac{d}{dt}{z}_1(L) > 0$, 
contradicting convergence; thus $L = 1$ and $z(t) \to e_1$.  
\\
Similarly, if $z_1(0) < z_1^{\dagger}$, then $z_1(t)$ is strictly decreasing, bounded below by $0$, 
and the same argument forces $z_1(t) \to 0$, i.e.\ $z(t) \to e_2$.  
\\
If $z_1(0) = z_1^{\dagger}$, then $\frac{d}{dt}{z}_1 = 0$ by \eqref{eq:N2-neg-eq}, so the trajectory stays at the interior unstable equilibrium. \qed}  
%
%
%
%
%
Now we have the following sufficient condition of initial state in the general case $N$-strain such that the strain with largest fitness excludes all other.
\begin{prop}\label{prop:neg-z0}
Consider the invader driven replicator system \eqref{eq:repli_invader} and assume that $0 > \lambda_1 \added{>} \cancel{\geq} \lambda_2 \geq \dots \lambda_N$. Let
$$
	C := \left\{ z \in \Delta:=\left\{z_1 + \dots + z_N = 1\right\} :\; \frac{z_1}{z_j} {\geq} \frac{\lambda_1}{\lambda_j}  \ \text{for all } j \geq 1 \right\} 
	.
	$$
\added{Then for almost everywhere such $\left(\lambda_1,\lambda_2,\dots,\lambda_N\right)$, $z(t) \to e_1$ for almost everywhere $z(0) \in C$. }
\deleted{ If $z(0)\in C,
$ then $z(t)\to e_1$.}
\end{prop}
\begin{proof}
\added{According to Lemma \ref{lmm:0-measure} and Proposition \ref{prop:a.e.vertexconverge}, we only need to consider when $\left(\lambda_1,\lambda_2,\dots,\lambda_N\right) \in \left(-\infty, 0\right)^N\setminus E$, with $E$ defined in Lemma \ref{lmm:0-measure}.}
Setting $R_{ij}:= \ln\frac{z_i}{z_j}$ for all $i,j$. 
On the hypersurface $\frac{z_i}{z_j} = \frac{\lambda_i}{\lambda_j}$,  
	we have
\begin{equation} \label{eq:neg-Rdot}
\frac{d}{dt}{R}_{ij} 
= 
\lambda_i\left(1 - z_i\right) - \lambda_j\left(1- z_j\right)
=
\left(\lambda_i - \lambda_j\right)\left(1 - z_i - z_j\right).
\end{equation}
Indeed, $\frac{z_i}{z_j} = \frac{\lambda_i}{\lambda_j}$ is equivalent to $\lambda_j \left(z_i - z_j\right) = z_j\left(\lambda_i - \lambda_j \right)$ since $\lambda_j,\lambda_i < 0$, we obtain that
$$
\lambda_i z_i - \lambda_j z_j 
= 
(\lambda_i - \lambda_j)z_i  + \lambda_j (z_i -  z_j) 
=
(\lambda_i - \lambda_j)z_i  + z_j\left(\lambda_i - \lambda_j \right)
=
(\lambda_i - \lambda_j) (z_i + z_j).
$$
Now we prove that $C$ is forward invariant.
\\ \\
$\bullet$ If any third strain is present,
then $z_1(t ) + z_j(t ) < 1$ and \eqref{eq:neg-Rdot} yields $\frac{d}{dt}{R}_{1j}(t ) {>} 0$ on the hypersurface $\frac{z_1}{z_j} = \frac{\lambda_1}{\lambda_j}$.  
The trajectory cannot cross from $R_{1j} {\ge} \ln\frac{\lambda_{1}}{\lambda_j}$ (since $z(0) \in C$) to $R_{1j} {<} \ln\frac{\lambda_{1}}{\lambda_j}$.  Indeed, assuming the contradiction, that then there exists $t > 0$  such that $R_{1j}\left(t\right) {<} \ln\frac{\lambda_1}{\lambda_j}$. By the continuity of $R_{1j}$, we can assume that $t_0 = \inf\left\{t>0,\; R_{1j}(t) {<} \ln\frac{\lambda_1}{\lambda_j}\right\}$, this leads to $R_{1,j} = \ln\frac{\lambda_1}{\lambda_j}$, then $\frac{d}{dt}{R}_{1j}\left(t_0\right) > 0$. 

On the other hand, since $R_{1j}\left(t_0\right) = \ln\frac{\lambda_1}{\lambda_j}$ and the trajectory crosses from $R_{1j} {\ge} \ln\frac{\lambda_{1}}{\lambda_j}$ to $R_{1j} < \ln\frac{\lambda_{1}}{\lambda_j}$, then  $\frac{d}{dt}{R}_{1j}\left(t\right) < 0$ for all $t\in \left(t_0 -\epsilon, t_0 +\epsilon\right)$ for some $\epsilon >0$ small enough
because $\frac{d}{dt}{R}_i$ is continuous. This is absurd as we have $R_{1j} = \ln\frac{\lambda_1}{\lambda_j}$ always implies $\frac{d}{dt}{R}_{1j} > 0$.
\\
\\
$\bullet$ If $z_1(t ) + z_j(t ) = 1$ (the $\{1,j\}$–face), then \eqref{eq:neg-Rdot} gives $\frac{d}{dt}{R}_{1j}(t ) = 0$.  

On that 2–face, the Theorem \ref{thm:neg-2strains} leads to the exclusion of strain $j$.
\\ \\
Since $j$ was arbitrary, we have that
\begin{equation} \label{eq:neg-Rratio}
R_{1j}(t) \;\ge\; \ln\frac{\lambda_1}{\lambda_j} 
\qquad \forall\, t \ge 0,\ \forall\, j \ge 2,
\end{equation}
implying $C$ is forward invariant. 

On the other hand, the Proposition \ref{prop:a.e.vertexconverge} claim the exclusion of $N-1$ species \added{for almost everywhere initial state}. Moreover, if $z(t) \to e_j$ for some $j \neq 1$, then $R_{1j} \to -\infty$, a contradiction with \eqref{eq:neg-Rratio}. Therefore $z(t) \to e_1$ as $t \to \infty$ \added{for almost everywhere $z(0) \in C$}.

\end{proof}
Accordingly, we have the following result in a particular case.
\begin{cor} \label{cor:neg-1/N}
When initial states are the same in all species \added{with invasion fitnesses satisfying $0 > \lambda_1 =\dots = \lambda_p:=\lambda_{\max} > \lambda_{p+1} \geq \dots \lambda_N$}, \replaced{the set of survivors consists of the species with greatest invader fitness $\lambda_{\max}$ }{ the unique survivor is the strain with greatest invader fitness}.
\end{cor}

\begin{proof}
\added{Firsly, we prove that, only top-fitness species can survive. Indeed, since $z(0) = \left(\frac{1}{N},\frac{1}{N},\dots, \frac{1}{N}\right) \in C$, by the forward-invariance established in the proof of Proposition \ref{prop:neg-z0}, $z_1(t)\ge z_j(t)$ for all $t\ge 0$ and all $j>p$.}

\added{On the other hand, the trajectory converges to an equilibrium $z^*$ by Theorem \ref{thm:neg-stable}. Passing the inequality to the limit, $z_1^*\ge z_j^*$ for all $j>p$, and in particular $z_1^*>0$ and $1\in S:=\{i:z_i^*>0\}$.}

\added{Suppose $j\in S$ with $j > p$, the equilibrium condition gives $\lambda_1(1-z_1^*)=\lambda_j(1-z_j^*)$. But $z_1^*\ge z_j^*$ gives $1-z_1^*\le 1-z_j^*$, and $\lambda_1>\lambda_j$ gives $\lambda_1(1-z_1^*) \neq \lambda_j(1-z_j^*)$, a contradiction. Hence $j \notin S$ for all $j > p$.}
\\ \\
\added{Secondly, we prove that the surviving species are in equal proportion. Take any two species $i,m\in\{1,\ldots,p\}$, so $\lambda_i=\lambda_m=\lambda_{\max}$. Consider
$R_{im}(t)=\ln \frac{z_i(t)}{z_m(t)}$ and we have that}
$$
\begin{aligned}
\frac{d}{dt} R_{im}
=
\lambda_i(1-z_i)-\lambda_m(1-z_m)
&=
\lambda_{\max}\bigl[(1-z_i)-(1-z_m)\bigr]
=
\lambda_{\max}(z_m-z_i)
\\
&
=
\lambda_{\max}z_m\left(1- \frac{z_i}{z_m}\right)
=
\lambda_{\max}z_m(1-e^{R_{im}})
,
\end{aligned}
$$
so
$$
\frac{d}{dt} R_{im}
=
\lambda_{\max}z_m(1-e^{R_{im}}).
$$
\added{It is trivial that $R_{im}=0$ is an equilibrium of this equation and recalling $R_{im}(0)=0$ since $z_i(0)=z_m(0)=1/N$.
By uniqueness of solutions of the ODE (the right-hand side is smooth in $R_{im}$ along the trajectory, as $z_m(t)>0$ on the interior), the constant solution $R_{im}(t)\equiv 0$ is the only solution with $R_{im}(0)=0$. Therefore}
$$
z_i(t)=z_m(t)
\qquad
\text{for all } t\ge 0,\ \text{for every pair } i,m\in\{1,\ldots,p\},
$$
yielding to the conclusion of Corollary \ref{cor:neg-1/N}.

\end{proof}

Moreover, from the proof of Proposition \ref{prop:neg-z0}, we also deduce the following necessary condition.
\begin{cor}(Necessary initial dominance to beat a fitter species)
\deleted{In the invader driven replicator dynamics with all negative fitnesses $\left(\lambda_1,\lambda_2,\dots,\lambda_N\right)$, if the unique survivor $i > 1$ is not the species with greatest invader fitness, then $\frac{z_j(0)}{z_i(0)} < \frac{\lambda_j}{\lambda_i}$ for all $j < i$. }
\added{ Fix $\lambda\notin E$, $E$ is defined in Lemma \ref{lmm:0-measure}, with $0>\lambda_1\ge\cdots\ge\lambda_N$. If a trajectory converges to a monomorphic equilibrium $e_i$ with $i\neq 1$, then $z(0)\notin C$; equivalently, there exists $j$ with
$\frac{z_j(0)}{z_1(0)}>\frac{\lambda_j}{\lambda_1}$.}
\end{cor}

To conclude this section, we have the following simulation for a general case, showing that we can have all alternatives of unique survivor.
\begin{figure}[htb!]
\centering
\includegraphics[width=0.65\linewidth]{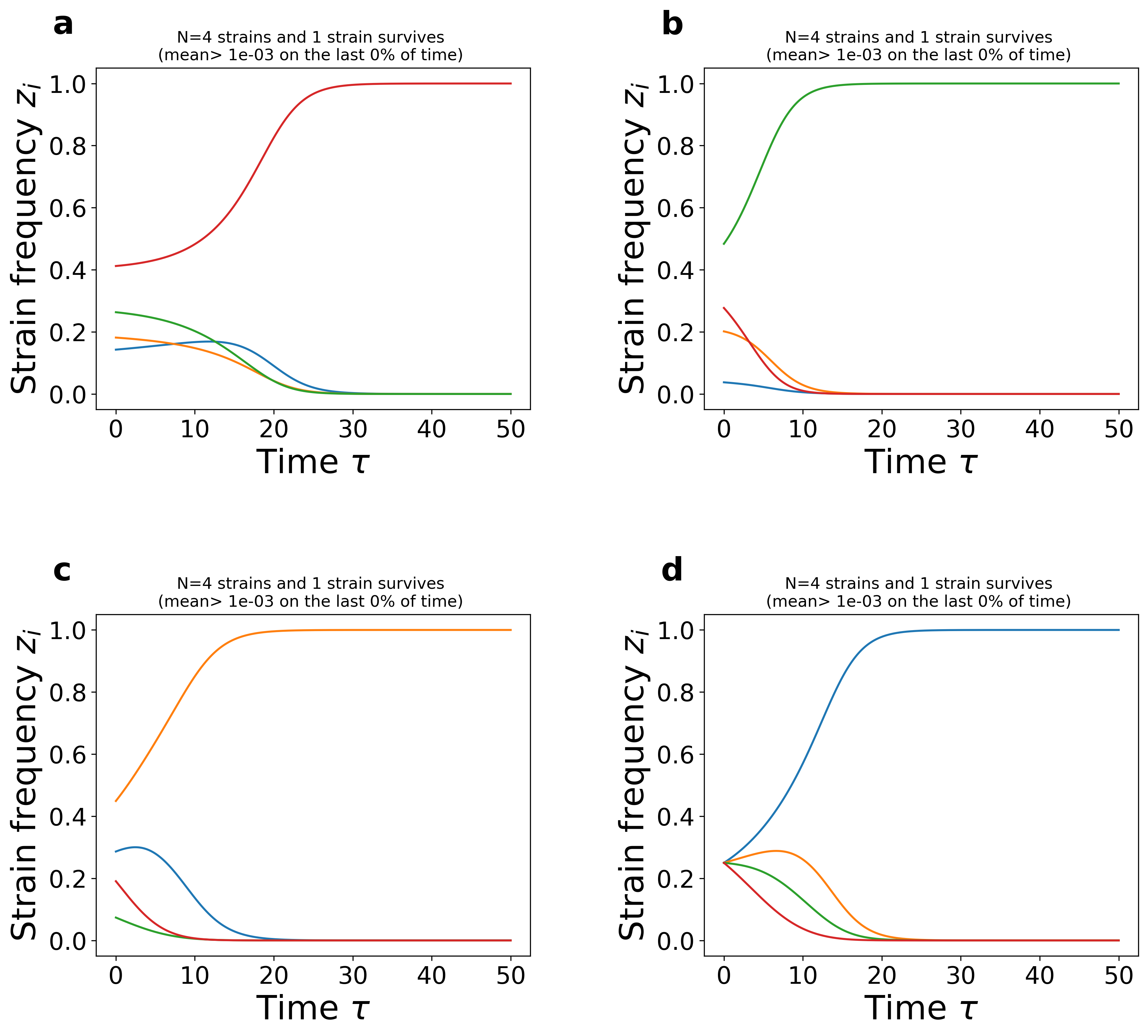}
\caption{\textbf{Initial conditions select the boundary attractor (winning species) in the negative invader-driven replicator system.} We simulate the dynamics of invader-driven replicator with $N=4$ and completely negative invasion fitness vector $\left(\lambda_1, \lambda_2, \lambda_3, \lambda_4\right)=(-0.4,-0.45,-0.5,-0.6)$. Panels (a–d) use different initial states $z(0)$. In the case of all negative fitness $\lambda_i<0$, the system is multistable: each run converges to a different boundary equilibrium $e_i$ (a unique survivor), showing that the surviving species depends on the initial composition.}
\label{fig:negative}
\end{figure}

\subsection*{Extension to the Non-Positive Invasion Fitness Case}
\added{In this subsection, we consider the more general case in which at least one fitness value equals 0. Recalling the definition of the simplex $\Delta$ in Proposition \ref{prop:neg-z0}, we obtain the following result.}
\begin{thm}
\added{Consider \eqref{eq:repli_invader} of $N \geq 1$ species and $0= \lambda_1 = \dots = \lambda_p > \lambda_{p+1} > \dots > \lambda_N$ . Assume that $\mathbf{z}(0)$ belongs to the interior of $\Delta$, i.e.: $z_i(0)>0$ for each $i\in\{1,\cdots,N\}$.
Denote $\Gamma_p = \left\{ z \in \Delta:\,  z_j = 0,\forall j>p \right\}$. 
\\
Then all the points on $\Gamma_p$ are equilibrium and the $\mathbf{z}(t)$ converges to a point in $\Gamma_p$ which depends on the initial state.}
\end{thm}
\begin{proof}
\added{ Set $\tilde
V(z) := \sum_{i=p+1}^N z_i$, implying $0\leq \tilde V(z) \leq 1$, then we have}
$$
\added{\frac{d}{dt}{\tilde V}(z) = \sum_{i=p+1}^N \frac{d}{dt}{z}_i 
= 
\sum_{i=p+1}^N z_i \left(\lambda_i (1-z_i) -Q(z)\right) 
= Q(z) (1- \tilde V(z))
\leq 0,} $$
\added{since}
$$\added{Q(z) = \sum_{j=p+1}^N \lambda_j(1-z_j)z_j \leq 0\,.}
$$
\added{Define  $\Omega := \left\{z \in \Delta : \frac{d}{dt}{\tilde V}(z) = 0\right\}=Q^{-1}(0)\cup \bar{V}^{-1}(1)$. 
Let us note that $Q^{-1}(0)=\Gamma_p\cup \mathcal{M}_p$. Where $\mathcal{M}_p$ is the set of each monomorphic state $j$ for $j>p$.
Each monomorphic steady state $\mathbf{e}_j \in\mathcal{M}_p$  satisfies $\tilde{V}(\mathbf{e}_j)=1$.
It follows that 
$$\Omega=\Gamma_p\cup \bar{V}^{-1}(1).$$
Clearly both $\Gamma_p$ and $\bar{V}^{-1}(1)$ are invariant, then 
 by the Lassalle's Invariance Principle, for any initial condition, the $\omega$-limit set  belongs to $\Gamma_p\cup \bar{V}^{-1}(1)$ as $t \to \infty$ (In other words, $z(t)$ asymptotically approach  this set).
Now, since $\mathbf{z}(0)$ belongs to the interior of $\Delta$, we have $\tilde{V}(0)<1$. 
The computation above shows that $\tilde{V}$ is decreasing in $t$ then $\tilde{V}(t)\leq \tilde{V}(0)<1$ for each $t\geq 0$. This excludes the case $\tilde V (z) = 1$ and it follows that the $\omega$-limits set belongs to $\Gamma_p$.}
\\
\\
\added{
Finally, it is clear that any point of $\Gamma_p$ is a steady state and thus trivially reachable by the dynamics. More can be say: any point of $\Gamma_p$ may be reach from the interior of $\Delta$. Indeed, take arbitrarily $i,m \leq p$ and consider the ratio $\frac{z_i(t )}{z_m(t )}$. It comes}
$$
\added{\frac{d}{dt } \left( \frac{z_i}{z_m} \right) 
 = 
 \frac{\frac{d}{dt}{z}_i z_m - z_i \frac{d}{dt}{z}_m}{z_m^2}
= 
\frac{(-z_i Q (z)) z_m - z_i (-z_m Q(z) )}{z_m^2}
= 
\frac{-z_i z_m Q(z) + z_i z_m Q(z)}{z_m^2} = 0\,.}
$$
\added{Combining with the previous arguments, 
as $t  \to \infty$, $z(t ) \to z^* \in \Gamma_p$ satisfying $\frac{z^*_i}{z^*_m} = \frac{z_i(0)}{z_m(0)}$ for all $ 1\leq i,m \leq p$, which end the proof.}

\end{proof}

\section{Simulations and numerical tests}
\subsection{Probability calculations for the number of coexisting species $n$ for different species pool size $N$}
In Figure \ref{fig:10strains_probs} \added{of the main text}, the explicit probabilities are computed by the formula in Theorem \ref{thm:prob} by Monte-Carlo method and are as follows.
\begin{align*}
	N=10,\; k=2:  &\quad P \approx 0.006568 \;\; \pm\; 0.000013 \quad (95\% \ \text{CI}) \\
	N=10,\; k=3:  &\quad P \approx 0.148881 \;\; \pm\; 0.000455 \quad (95\% \ \text{CI}) \\
	N=10,\; k=4:  &\quad P \approx 0.392741 \;\; \pm\; 0.002090 \quad (95\% \ \text{CI}) \\
	N=10,\; k=5:  &\quad P \approx 0.319676 \;\; \pm\; 0.003288 \quad (95\% \ \text{CI}) \\
	N=10,\; k=6:  &\quad P \approx 0.111410 \;\; \pm\; 0.002456 \quad (95\% \ \text{CI}) \\
	N=10,\; k=7:  &\quad P \approx 0.018318 \;\; \pm\; 0.000496 \quad (95\% \ \text{CI}) \\
	N=10,\; k=8:  &\quad P \approx 0.001573 \;\; \pm\; 0.000208 \quad (95\% \ \text{CI}) \\
	N=10,\; k=9:  &\quad P \approx 0.000067 \;\; \pm\; 0.000023 \quad (95\% \ \text{CI}) \\
	N=10,\; k=10: &\quad P \approx 0.000002 \;\; \pm\; 0.000002 \quad (95\% \ \text{CI}) \,.\\
\end{align*}
For another example, we performed another simulation with $N=20$.
\begin{figure}[htb!]
\centering
\includegraphics[width=1\linewidth]{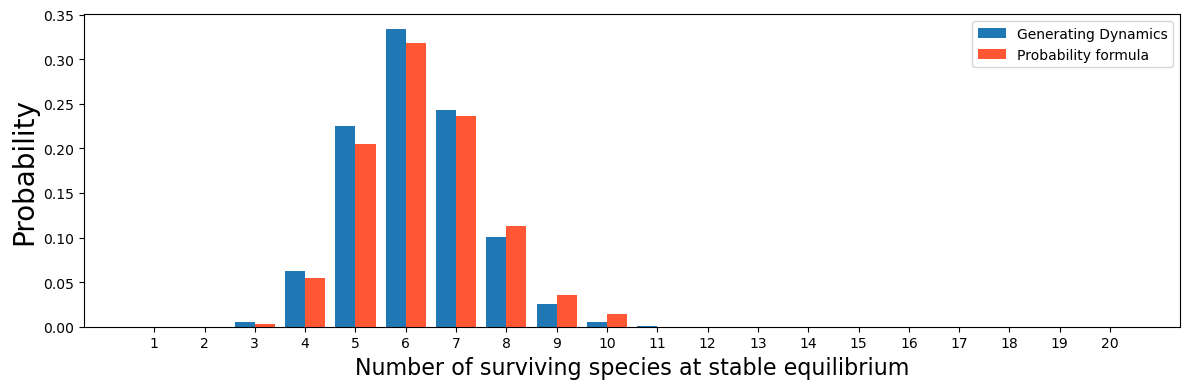}
\caption{\textbf{ Probability mass function of the number $k$ of surviving species at the unique 
	stable equilibrium of the invader-driven replicator \eqref{eq:repli_invader} with $N = 20$ and i.i.d. invader fitnesses
	$\lambda_i \sim \mathcal{U}[0,1]$.}
	Blue bars: empirical frequencies from 10,000 independent ODE runs started in the simplex interior 
	(survival counted when $z_i(T) > 10^{-4}$).
	Orange bars: probabilities obtained by Monte-Carlo evaluation of the Theorem \ref{thm:prob} integral formula, 
	reported with 95$\%$ confidence intervals.
	The two estimates agree similar probabilities of the target events.}
\label{fig:20strains_probs}
\end{figure}
Both Figures \ref{fig:10strains_probs} and \ref{fig:20strains_probs} confirm the Theorem \ref{thm:prob}, and are consistent with Figure \ref{fig:nbar} presented in the main text.

\subsection{Testing whether an ecological multispecies system is invader-driven}\label{supp:test}
\paragraph{\textbf{Increasing statistical power}} We have another hypothesis-testing for the data in Section \ref{sec:sub-test}, using condition \eqref{eq:test-linear} for invader-driven replicator system, taking the Nepal as a reference site, because it has the largest number of serotypes and thus shares more serotypes with any other site in the dataset.
\begin{figure}[b!]
    \centering
    \includegraphics[width = 0.45 \textwidth]{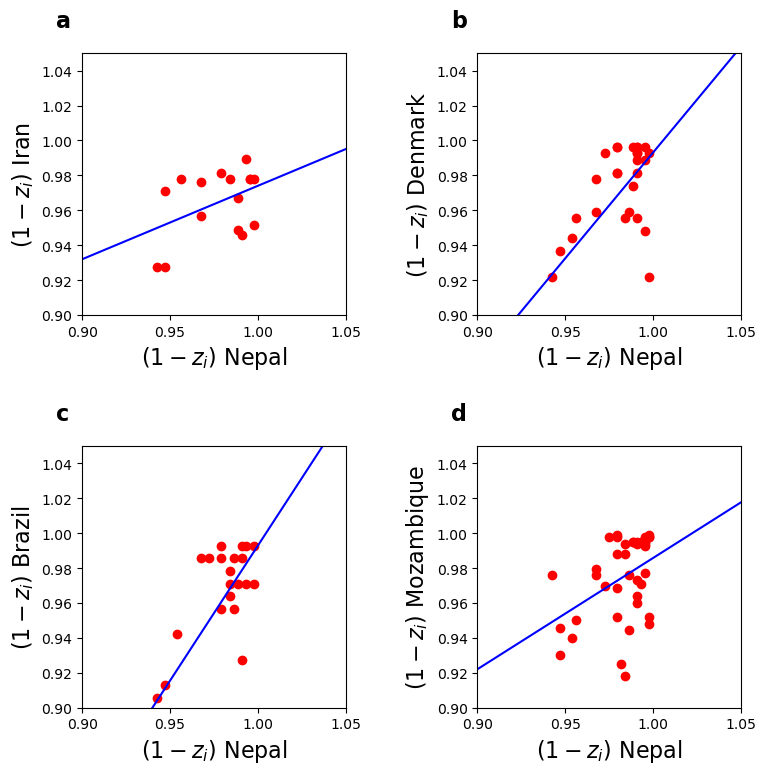}
    \caption{\textbf{Testing for invader-driven $\Lambda$ matrix in \textit{Streptococcus pneumoniae} with another country as reference.} Common serotypes in pairs of countries/epidemiological settings (data analyzed in \citep{le2025inference}) are shown by the red circles. In blue lines we show the corresponding linear fit plots for these comparisons. We take Nepal as the reference setting here because it has the higher number of serotypes, and compare the reported $1-z_i$ in each other setting to the value reported in Nepal (sets of overlapping $i$ can be different in each pair of sites). If the invader-driven $\Lambda$ matrix hypothesis is correct, there should be a proportionality constant linking these observations for common serotypes occurring in both settings. 
We do the test for these pairs: (Nepal, Iran), (Nepal, Denmark), (Nepal, Brazil) and (Nepal, Mozambique). The $p$-values for each linear regression are: a) $ 0.092$, b) $0.0008$, c) $0.0003$, and d) $0.0086$. The corresponding slopes are: $0.4206$, $1.2124$, $1.5472$ and $0.6373$. Even though three regressions seem significant, the standard deviations around these slopes are rather large, in particular in a, b, c, d respectively: $0.2323$, $0.3183$, $0.3629$ and $0.2296$, and the $R^2$ values modest: $0.1897$, $ 0.3581$, $0.4640$ and $0.1723$.
}
\label{fig:correlationPlots}
\end{figure}


Noting the $p$-values for the regression slopes lends support to the linearity hypothesis, in 3 out of 4 combinations, in line with invader-driven replicator expectations for $\left(1-z_i\right)$ in pairs of sites at equilibrium.

\paragraph{\textbf{Slope magnitudes between sites are not independent}} 
However, there is another subtle requirement involved in hypothesis-testing, using condition \eqref{eq:test-linear} for invader-driven replicator system, between trios of equilibrium sites. Their shared serotype frequency slopes are non-independent. Let us take an example the three pairs (Denmark, Mozambique), (Denmark, Brazil), and (Brazil, Mozambique).
\begin{figure}[htb!]
    \centering
    \includegraphics[width = 0.9 \textwidth]{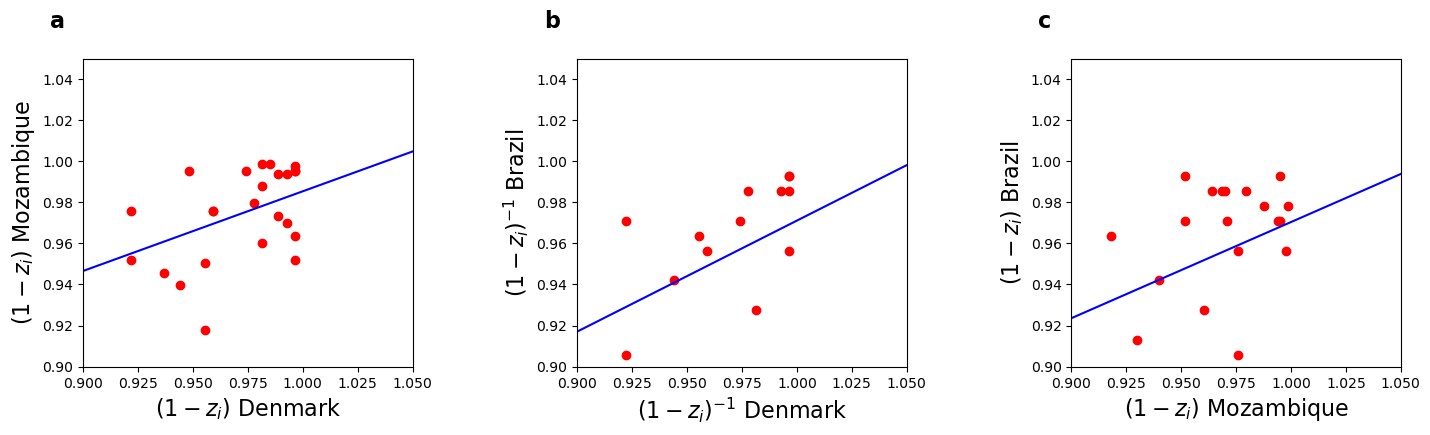}
    \caption{\textbf{Testing for invader-driven $\Lambda$ matrix in a multi-species respiratory pathogen: \textit{Streptococcus pneumoniae} accounting for slope magnitude constraints.}
We illustrate the test only for the trio Brazil, Denmark, Mozambique, involving these pairs: (Denmark, Mozambique), (Denmark, Brazil), and (Mozambique, Brazil). The $p$-values for the slopes respectively are: a) $p=0.003$(**), b) $p=0.035$, and c) $p=0.211$ (n.s). Among three pairs, in the cases of (Denmark, Mozambique) and (Denmark, Brazil), their $p$-value $0.003$ and $0.035$ allows us to favour the hypothesis then there may be a linear relation between $1-z_i$ of these pair. The slopes of these linear regressions respectively are: $0.388$, $0.540$, and $0.468$. (sd: $0.117$, $0.232$, and $0.362$). The R-squared values: $0.325$, $0.280$,  and $0.081$.
}
\label{fig:2nd-round-LinearTest}
\end{figure}


To conclude the invader-driven case, the slope between (Mozambique, Brazil) should be \textit{equal} to the ratio between slopes (Denmark, Brazil) and (Denmark, Mozambique). 
\begin{align}
\frac{1-z_i^{(\text{Mozambique})}} {1-z_i^{(\text{Denmark})}}
&
=
\underbrace{\frac{Q^{(\text{Mozambique})}}{Q^{(\text{Denmark})}}}_{s_1},
\\
\frac{1-z_i^{(\text{Brazil})}}{1-z_i^{(\text{Denmark})}}
&=
\underbrace{\frac{Q^{(\text{Brazil})}}{Q^{(\text{Denmark})}}}_{s_2},
\\
\frac{1-z_i^{(\text{Brazil})}}{1-z_i^{(\text{Mozambique})}}
&
=
\underbrace{\frac{Q^{(\text{Brazil})}}{Q^{(\text{Mozambique})}}}_{=s_2/s_1?}
\end{align}

The slopes of these linear regressions, as shown in Figure \ref{fig:2nd-round-LinearTest}, respectively are: $s_1=0.388$, $s_2=0.540$, and $s_3=0.468$. Yet, we can easily verify that $s_3 \neq s_2/s_1.$
Thus, in this particular case, such test also fails to lend support to the hypothesis of invader-driven replicator dynamics underlying serotype coexistence in pneumococcus, first because slopes are not always significant, and also because their magnitudes do not satisfy proportionality constraints.

\paragraph{\textbf{Regression intercept must be zero for invader-driven hypothesis}} 
This is an obvious requirement from the invader-driven replicator \eqref{eq:test-linear}, that yet needs to be checked. When considering the site pairs in Figure \ref{fig:correlationPlots}, the $p$-values for the $z_i$ regression intercepts  (Nepal, Denmark), (Nepal, Brazil) and (Nepal, Mozambique) are $0.488$, $0.133$ and $0.131$, respectively. This confirms that these intercepts are zero in all these pairs, with the exception of (Nepal, Iran) where the intercept of the regression is significantly non-zero.

Meanwhile, in Figure \ref{fig:2nd-round-LinearTest}, the $p$-values for the regression intercepts  (Denmark, Mozambique), (Denmark, Brazil) and (Mozambique, Brazil) are $2.23 \times 10^{-5}$, $0.073$ and $0.169$, respectively, indicating rejection of the null hypothesis in 2 out of 3 pairs (Denmark, Mozambique) and (Denmark, Brazil).
\\

\section{\added{Proof of auxiliary results in Section \ref{subsec:mean-NoSpecies} for $\mathbb E[n]$ asymptotic derivation}}\label{app:proofs}
\begin{change}
For simplicity, we  use the following notation.
\begin{defn}[Notation $\mathcal O_{\mathbb P}(\alpha_N)$]\label{def:Op}
Let $(X_N)_{N\ge 1}$ be a sequence of random variables and 
$(\alpha_N)_{N\ge 1}$ a sequence of positive real numbers.  
We write
\[
X_N = \mathcal O_{\mathbb P}(\alpha_N)
\]
if there exists a constant $M>0$ such that, for every $\varepsilon>0$,
\[
\mathbb P\big(|X_N| > M\,\alpha_N\big) \xrightarrow[N\to\infty]{} 0.
\]
Equivalently, the sequence $\left\{\frac{X_n}{\alpha_N}\right\}_{N\to \infty}$ is bounded in probability.
\end{defn}

With this notation, we can now prove the Lemmas \ref{lmm:concentrated} and \ref{lmm:determinEn}. 
\begin{proof}[Proof of Lemma \ref{lmm:concentrated}]
By Lemma~\ref{lmm:PDF-Beta}, for each $1\le j\le k\le C\sqrt N$, $\lambda_j \sim \mathrm{Beta}(N+1-j,\; j)$. Hence
\[
 \mathbb E[\lambda_j] = \frac{N+1-j}{N+1}:=u_j,
\quad
\text{and}
\quad\;
\mathrm{Var}(\lambda_j)
= \frac{(N+1-j)\,j}{(N+1)^2(N+2)}
= \mathcal O(N^{-3/2}),
\]
since $j\le C\sqrt N$. Therefore,  Bienaym\'e-Tchebychev inequality yields:
\[
\lambda_j -  u_j = \mathcal O_{\mathbb P}(N^{-3/4}).
\]
Next, for $k\le C\sqrt N$,
\[
\lambda_j^{-1} - u_j^{-1}
= \mathcal O_{\mathbb P}(N^{-3/4}),
\]
so summing over $1\le j\le k$ gives
\[
\sum_{j=1}^k \lambda_j^{-1}
= \sum_{j=1}^k  u_j^{-1} + \mathcal O_{\mathbb P}(kN^{-3/4})
= \sum_{j=1}^k  u_j^{-1} + \mathcal O_{\mathbb P}(N^{-1/4}),
\]
because $k=O(\sqrt N)$.
Since $ u_j^{-1}\in[1,2]$ for all $j\le k$, we have
$\sum_{j=1}^k  u_j^{-1} = k + \mathcal O(1)$.
Writing
\[
A_k := \sum_{j=1}^k  u_j^{-1},
\qquad
\Delta_k := \sum_{j=1}^k (\lambda_j^{-1} -  u_j^{-1}),
\]
we have $A_k = k + \mathcal O(1)$ and $\Delta_k = \mathcal O_{\mathbb P}(N^{-1/4})$. Thus
\[
Q^*_k -  S_k
= (k-1)\left(\frac{1}{A_k+\Delta_k} - \frac{1}{A_k}\right)
= (k-1)\frac{-\Delta_k}{A_k(A_k+\Delta_k)}
= \mathcal O_{\mathbb P}(N^{-3/4}),
\]
Since $A_k = k + \mathcal O(1)$ and $k=O(\sqrt N)$, we have
\[
A_k(A_k+\Delta_k) = k^2 + \mathcal O_{\mathbb P}(kN^{-1/4})
= N + \mathcal O_{\mathbb P}(N^{3/4}),
\]
and therefore $Q^*_k -  S_k
= \mathcal O_{\mathbb P}(N^{-3/4})$.

\end{proof}

\begin{proof}[Proof of Lemma \ref{lmm:determinEn}]

\medskip
\noindent\textit{(i) Existence and uniqueness.}
This is a  particular case of  Proposition \ref{rmk:prop-k} with $\lambda_j=u_j$ and $Q_k^*=S_k$.

\medskip
\noindent\textit{(ii) Asymptotics.}
We now compute the asymptotic behaviour of $\widehat n$ as $N\to\infty$. Observe that $u_j^{-1} = \frac{N+1}{N+1-j}$ for all $1\leq j \leq N$,
so
\[
\sum_{j=1}^k u_j^{-1}
= (N+1)\sum_{j=1}^k \frac{1}{N+1-j}
= (N+1)\sum_{m=N+1-k}^{N} \frac{1}{m}.
\]
Let $H_n:=\sum_{m=1}^n \frac{1}{m}$ denote the $n$-th harmonic number. Then
\[
\sum_{m=N+1-k}^{N} \frac{1}{m}
= H_N - H_{N-k}
= \ln\frac{N}{N-k} + O\!\left(\frac{1}{N}\right),
\]
uniformly for $k=O(\sqrt N)$. Hence
\[
\sum_{j=1}^k u_j^{-1}
= (N+1)\left(\ln\frac{N}{N-k} + O\!\left(\frac{1}{N}\right)\right).
\]
For $k=O(\sqrt N)$ we have
\[
\ln\frac{N}{N-k}
= \ln\Bigl(1 + \frac{k}{N-k}\Bigr)
= \frac{k}{N} + \frac{k^2}{2N^2} + O\!\left(\frac{k^3}{N^3}\right),
\]
so
\[
\sum_{j=1}^k u_j^{-1}
= (N+1)\left(\frac{k}{N} + \frac{k^2}{2N^2} + O\!\left(\frac{k^3}{N^3}\right)\right)
= k + \frac{k^2}{2N} + O\!\left(\frac{k^3}{N^2}\right).
\]
Therefore
\[
S_k
= (k-1)\Bigl(\sum_{j=1}^k u_j^{-1}\Bigr)^{-1}
= (k-1)\left(k + \frac{k^2}{2N} + O\!\left(\frac{k^3}{N^2}\right)\right)^{-1}.
\]
For $k=O(\sqrt N)$ this yields
\[
S_k
= \frac{k-1}{k}\left(1 + \frac{k}{2N} + O\!\left(\frac{k^2}{N^2}\right)\right)^{-1}
= \left(1 - \frac{1}{k}\right)\left(1 - \frac{k}{2N} + O\!\left(\frac{k^2}{N^2}\right)\right),
\]
so
\[
S_k
= 1 - \frac{1}{k} - \frac{k-1}{2N} + O\!\left(\frac{k^2}{N^2}\right).
\]
On the other hand,
\[
u_k = 1 - \frac{k}{N+1}
= 1 - \frac{k}{N} + O\!\left(\frac{k}{N^2}\right),
\qquad
u_{k+1} = 1 - \frac{k+1}{N} + O\!\left(\frac{k}{N^2}\right).
\]
The inequalities
\[
u_{k+1} \le S_k \le u_k
\]
are thus equivalent, for $k=O(\sqrt N)$, to
\[
1 - \frac{k+1}{N} \;\le\; 1 - \frac{1}{k} - \frac{k-1}{2N} + O\!\left(\frac{k^2}{N^2}\right)
\;\le\; 1 - \frac{k}{N} + O\!\left(\frac{k}{N^2}\right).
\]
Subtracting $1$ and multiplying by $-N$ gives
\[
k+1 \;\ge\; \frac{N}{k} + \frac{k-1}{2} + O\!\left(\frac{k^2}{N}\right)
\;\ge\; k.
\]
Hence $k$ satisfies
\[
k^2 = 2N + O(\sqrt N),
\]
and therefore
\[
\widehat n \sim \sqrt{2N}
\quad\text{as }N\to\infty.
\]

\end{proof}

\end{change}

\end{document}